\documentclass[a4paper,12pt]{article}
\usepackage{fullpage}

\usepackage{cmbright}
\usepackage{microtype}
\usepackage[T1]{fontenc}
\usepackage{lmodern}
\usepackage[utf8]{inputenc}

\usepackage{bm} 
\usepackage{amsfonts,amssymb,amsmath,amsthm,dsfont} % amsthm
\usepackage{mathbbol}
\DeclareSymbolFontAlphabet{\mathbb}{AMSb}
\DeclareSymbolFontAlphabet{\mathbbl}{bbold}
\usepackage{xspace}% pour définir les commandes \ie et \eg
\usepackage{graphicx}
\usepackage{xcolor}
\usepackage{subfig}
\usepackage{comment}

\usepackage{hyperref}
\definecolor{darkgreen}{rgb}{0,0.4,0}
\definecolor{BrickRed}{rgb}{0.65,0.08,0}
\hypersetup{colorlinks=true,linkcolor=blue,citecolor=red,filecolor=BrickRed,urlcolor=darkgreen}

\usepackage[capitalize,nameinlink,noabbrev]{cleveref}

\usepackage{tikz}
\usetikzlibrary{arrows.meta,decorations.markings}
\usepackage{tkz-tab}
\usepackage{multirow}

\usepackage[backend=bibtex,block=space,sortcites=true,maxbibnames=99,style=numeric]{biblatex}
\usepackage{doi}
\renewbibmacro{in:}{}
\DeclareFieldFormat[article]{volume}{\mkbibbold{#1}}
\DeclareFieldFormat[article]{number}{(#1)}
\DeclareFieldFormat{doi}{\href{https://dx.doi.org/#1}{\textsc{doi}}}
\AtEveryBibitem{\clearfield{issn}\clearfield{isbn}\clearlist{language}}
\renewbibmacro*{volume+number+eid}{%
  \printfield{volume}%
  \printfield{number}%
  \setunit{\bibeidpunct}%
  \printfield{eid}}
\def\abx@missing@entry#1{\abx@missing{#1??}}
\usepackage{booktabs}
\usepackage{array}
\newcolumntype{C}{>{$}c<{$}} % centered math column
\newcolumntype{L}{>{$}l<{$}} % left math column
\usepackage{todonotes}
\usepackage{comment}

\newtheorem{theorem}{Theorem}
\newtheorem*{theorem*}{Theorem}
\newtheorem{lemma}[theorem]{Lemma}
\newtheorem{proposition}[theorem]{Proposition}

\newtheorem{condition}[theorem]{Condition}
\newtheorem*{main-thm}{Main Theorem}
\crefname{assumption}{Assumption}{Assumptions}
\crefname{condition}{Condition}{Conditions}

\theoremstyle{definition}

\newtheorem{remark}[theorem]{Remark}

\newtheorem{conjecture}[theorem]{Conjecture}

\newcommand{\eps}{\varepsilon}
\newcommand{\ZZ}{\mathbb{Z}}

\newcommand{\EE}{\mathbb{E}}

\newcommand{\LandauO}{\mathcal{O}}

\newcommand{\OEIS}[1]{\href{http://oeis.org/#1}{OEIS~#1}}

\renewcommand{\P}{\mathbb{P}}

\newcommand{\E}{E}

\newcommand{\nto}{\ensuremath{n\rightarrow\infty}}
\newcommand{\largev}{\ensuremath{n-2g =\omega(\log n)}}
\newcommand{\smallv}{\ensuremath{n-2g =o(\log n)}}

\newcommand{\Sij}{\sum_{i,j\in\mathcal S}}

\newcommand{\f}[2]{f\left(\frac{#1}{#2}\right)}

\newcommand{\bigO}[1]{O\left({#1}\right)}
\newcommand{\bigTheta}[1]{\Theta\left({#1}\right)}
\newcommand{\smallo}[1]{o\left({#1}\right)}
\newcommand{\halfminth}{\left(\frac 1 2 -\theta\right)}
\newcommand{\err}{\bigO{n^{-1}\log^{-2}(n)}}

\newcommand{\bigpar}[1]{\left({#1}\right)}

\newcommand{\goodB}{\mathcal{B}^{\text{good}}}
\newcommand{\goodBB}{\overline{\goodB}}
\newcommand{\badB}{\mathcal{B}^{\text{bad}}}
\newcommand{\badBB}{\overline{\badB}}
\newcommand{\good}[1]{\mathcal{G}^{\text{good}}({#1})}
\newcommand{\goodo}[1]{\overline{\mathcal{G}^{\text{good}}}({#1})}
\newcommand{\goodG}{\good{n}}
\newcommand{\goodGG}{\overline{\goodG}}
\newcommand{\goodGNT}{\good{N_\tau}}

\newcommand{\goodGNK}{\good{N_k}}
\newcommand{\bad}[1]{\mathcal{G}^{\text{bad}}({#1})}
\newcommand{\bado}[1]{\overline{\mathcal{G}^{\text{bad}}}({#1})}
\newcommand{\badG}{\bad{n}}
\newcommand{\badGG}{\overline{\badG}}

\newcommand{\badGNK}{\bad{N_k}}
\newcommand{\interv}{\mathcal I(n)}
\newcommand{\gstrict}{\ensuremath{g\in \interv}}
\newcommand{\mstrict}{\ensuremath{(n,g)\in \mathcal I(m)}}
\newcommand{\badV}{\log^{4/3}(n)}

\def \l{\lambda}
\def\ll{\log(\lambda)}
\def\lll{\log(-\ll)}

\newcommand{\cball}{\overline{B}}

\providecommand{\keywords}[1]
{
  \small	
  \noindent
  \textbf{{Keywords:}} #1
}

\begin{document}

%% define your title in the usual way
\title{\textbf{Bivariate asymptotics via random walks: \\application to large genus maps}}

\author{%
  Andrew Elvey Price\thanks{CNRS and Institut Denis Poisson, Université de Tours, France}
  \and
  Wenjie Fang\thanks{Univ Gustave Eiffel, CNRS, LIGM, F-77454 Marne-la-Vallée, France}
  \and
  Baptiste Louf\thanks{CNRS and Institut de Mathématiques de Toulouse, France}
  \and
  Michael Wallner\thanks{Institute of Discrete Mathematics, TU Graz, Austria and Institute of Discrete Mathematics and Geometry, TU Wien, Austria} 
}

\maketitle

\begin{abstract}
  We obtain bivariate asymptotics for the number of (unicellular) combinatorial maps (a model of discrete surfaces) as both the size and the genus grow. This work is related to two research topics that have been very active recently: multivariate asymptotics and large genus geometry.  
  Our method consists of studying a linear recurrence for these numbers, and can be applied to many other linear recurrences. In particular, we include a general theorem that yields asymptotics for such recurrences, provided that some assumptions are satisfied.
  \thispagestyle{empty}
\end{abstract}


%% ====================================%%

\section{Introduction} \label{sec:intro}

This paper deals with \emph{multivariate asymptotics} and \emph{large genus geometry}, two recent topics that have gained a lot of traction in the past 15 years \cite{Agg,linear-genus,budzinski-louf,DGZZ,acsv_web}.

\emph{Asymptotic enumeration} is one of the main branches of combinatorics, and it has many applications; see, e.g., \cite{flajolet}. The univariate case (when one parameter tends to infinity) has been studied extensively and many general methods have been developed to tackle it, but the \emph{multivariate} case is notoriously much harder. A recent systematic approach to this topic is analytic combinatorics in several variables (ACSV) \cite{acsv,acsv-mishna}. However, the ACSV theory mostly only applies to rational and algebraic generating functions, as it requires a good knowledge on multivariate singularities of the generating function. Another approach is to adapt univariate tools such as the saddle-point method \cite{flajolet} to the multivariate case by interpreting the numbers as coefficients of a sequence of univariate generating functions. This approach needs a refined analysis of the singularities of each univariate generating function in the sequence. Instead of looking at generating functions, there are also works on extracting asymptotic expressions from recurrence relations, see for instance \cite[Chapter~9]{rec-survey} and references therein. 

Recently, a subset of the current authors worked on the asymptotic enumeration of a combinatorial family called compacted trees \cite{compacted-tree}, which led to a new method for asymptotic enumeration through the analysis of \emph{bivariate recurrences} having a certain behaviour. This method was later applied to other objects such as minimal automata~\cite{minimal-automata}, Young tableaux~\cite{banderier2021walls} and phylogenetic networks~\cite{fuchs2021phylo}. Although this new method so far yields only univariate asymptotics, it gave us hope that some of its ideas may be adaptable to other combinatorial families governed by bivariate recurrences, including the extraction of bivariate asymptotics. %One source of such families is the study of combinatorial maps through the lens of integrable systems.

In this article, we introduce a new method for determining bivariate asymptotics, given a (linear) bivariate recurrence. As in \cite{compacted-tree}, the first ingredient in our method is an asymptotic guess and check approach, in which the bivariate recurrence transforms into a differential equation for a function involved in the asymptotic form. The second ingredient, which is new to this work, involves modelling the recurrence by a random walk. In particular, we use our new method to determine bivariate asymptotics for \emph{unicellular maps}, a type of combinatorial map.

\emph{Combinatorial maps}, or simply \emph{maps}, can be seen as a model of discrete surfaces or, alternatively, graphs on surfaces. It thus comes with two natural parameters: the size (usually the number of edges), and the genus of the surface. For a long time, the enumeration of combinatorial maps was confined to constant genus, for which general asymptotic enumeration, sophisticated exact expressions and even illuminating bijections exist, especially for planar maps (see \cite{planar-map-survey} for a survey). The study of asymptotic enumeration when the genus varies with the size has only been approached much more recently. This is partly due to the growing interest in understanding the geometry of {large genus }maps (see, e.g., \cite{budzinski-louf,DGZZ}), for which asymptotic enumeration proves to be a crucial tool. For instance, the analysis of large genus geometry of certain models in \cite{DGZZ} crucially relies on the asymptotics of related objects proven in \cite{Agg}. However, we lack general tools to deal with the question of bivariate asymptotic enumeration of maps. Existing works (see, e.g., \cite{budzinski-louf,meso-genus,linear-genus}) derive such asymptotics as corollaries of probabilistic results, which is not the direction of dependency that we want.

Combinatorial maps can be considered as models of certain integrable systems studied in theoretical physics (see, e.g., \cite{LandoZvonkine}). This perspective has led to surprising enumerative relations that are difficult to interpret combinatorially, including the Harer--Zagier recurrence for unicellular maps~\cite{HarerZagier1986Euler}, and the Goulden--Jackson recurrence for triangulations \cite{KP-triangulation}. The Harer--Zagier recurrence is precisely what we use in this article to determine bivariate asymptotics for unicellular maps.

As there is no well-established method for extracting \emph{bivariate} asymptotic behaviour from linear bivariate recurrences, a new method is needed, and by working on the aforementioned Harer--Zagier recurrence, we devise such a new method based on the asymptotic guess-and-check approach and random walks, with which we obtain the bivariate asymptotic enumeration of unicellular maps uniformly for all regimes; see \Cref{sec:main_results} for details. Although these results are already known in several regimes \cite{meso-genus,linear-genus}, our approach is \emph{structure-agnostic}, meaning that it does not use any prior knowledge on the combinatorial model, but relies solely on the bivariate recurrence. As a consequence, our new method has the potential to be applied to any other class of combinatorial objects governed by a similar recurrence.

In this paper, we compute the asymptotic number of unicellular maps in terms of two parameters (size and genus) by analysing a linear bivariate recurrence satisfied by these numbers. Our method relies on two main ingredients: an \emph{asymptotic guess and check} approach, and modelling the recurrence by a \emph{random walk}. An informal overview of how these two ingredients are combined in our method is given in \Cref{sec:guess}. The first step of the method: {\em guessing} an asymptotic enumeration of unicellular maps is given in \Cref{sec:guessing}. In \Cref{sec:RW} we give a detailed explanation of how to use random walks to validate our guessed asymptotic expression, which depends on several assumptions on random walks we construct from the bivariate recurrence. In \Cref{sec:largev} and \Cref{sec:smallv}, we check that the assumptions laid out in \Cref{sec:RW} are valid in both the lower regime $n - 2g = \omega(\log n)$ and the higher one $n - 2g = o(\log n)$, thus concluding the validity of our guessed asymptotic expression. For the intermediate regime $n - 2g = \Theta(\log n)$, our method does not apply, and we resort to the saddle-point method to check our result in this regime, thus establishing its validity for the whole parametric space. After showcasing our method on asymptotic enumeration of unicellular maps, we propose in \Cref{sec:RW_general} a general form of our asymptotic guess and check method that applies to linear bivariate recurrences satisfying several assumptions, which are more general than those laid out in \Cref{sec:RW}. Finally, in \Cref{sec:discussion}, we conclude with some discussion of our method, notably how it differs from existing works, its generality, and ongoing future work beyond linear recurrences.

\begin{comment}

  Harer--Zagier series formula~\cite[Proposition~7]{ChapuyFerayFusy2013Unicellular} and \cite[Theorem~3]{HarerZagier1986Euler}\footnote{Note that the correction term $2xy$ is missing.}:
  \begin{align*}
    1+2xy+2\sum_{g\geq 0, n>0}\frac{E(n,g)}{(2n-1)!!}y^{n+1}x^{n+1-2g} = \left(\frac{1+y}{1-y}\right)^{x}.
  \end{align*}

  Also, according to~\cite[Theorem~3]{HarerZagier1986Euler} the following recurrence holds: Let
  \begin{align*}
    1 + 2 \sum_{n \geq 0} c_n(x) y^{n+1} = \left(\frac{1+y}{1-y}\right)^{x},
  \end{align*}
  then the polynomials $c_{n}(x)$ satisfy
  \begin{align*}
    c_n(x) &= c_n(x-1) + c_{n-1}(x) + c_{n-1}(x-1),
  \end{align*}
  with boundary conditions $c_0(x)=x$, $c_{n}(0)=0$ for $n,x \geq 0$.
  
\end{comment}

\section{Main result on large genus unicellular maps}\label{sec:main_results}

A \emph{unicellular map} with $n$ edges and genus~$g$ is the combinatorial data of a $2n$-gon whose sides are identified two by two to form a compact, connected, oriented surface of genus~$g$, along with an additional distinguished oriented edge called the root; see \Cref{fig:map}.  Note that unicellular maps of genus $0$ are rooted plane trees.

\begin{figure}
  \centering
  \includegraphics[width=0.6\linewidth]{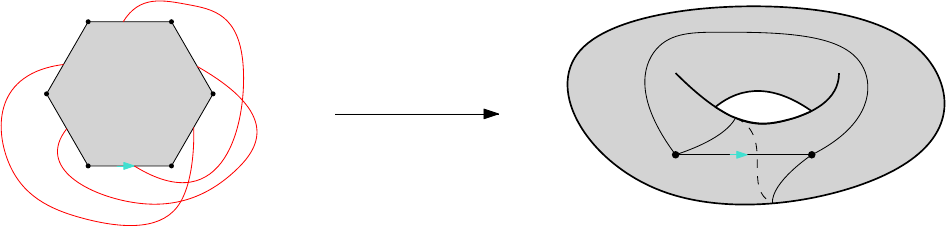}
  \caption{A unicellular map of genus $1$ with $3$ edges.}
  \label{fig:map}
\end{figure}

Let $\E(n,g)$ be the number of unicellular maps with $n$ edges and genus $g$. They  satisfy the \emph{Harer--Zagier recurrence formula}; see~\cite[p.~460]{HarerZagier1986Euler}:
\begin{align}
  \label{eq:recurrence}
  (n+1)\E(n,g)=2(2n-1)\E(n-1,g)+(n-1)(2n-1)(2n-3)\E(n-2,g-1),
\end{align}
with boundary conditions 
\begin{equation}\label{eq_bdy}
  \E(0,0)=1 \qquad \text{ and } \qquad  \E(n,g)=0\text{ if } g<0 \text{ or } n<2g.
\end{equation}

\begin{theorem}[Asymptotic enumeration of large genus unicellular maps]\label{thm_main}
  Given a sequence\footnote{We will use the notation $g \equiv g_n$ to denote the omission of the dependency of $g$ on $n$. 
    % We will omit the dependency of $g$ on $n$ from now on.
  } $g \equiv g_n$ such that $\frac{n-2g}{\log n}\rightarrow \infty$ as $n\to\infty$, the following asymptotics hold:
  \begin{equation}\label{eq_main_asympto}
    \E(n,g) \sim \frac{1}{2\sqrt{\pi}} \frac{\sqrt{g}(g/e)^g}{g!} n^{2g-2} e^{nf(\frac{g}{n})} {J}\left(\frac{g}{n}\right),
  \end{equation}
  with $f$ and $J$ defined as follows (see \Cref{fig:lambda-f-J}): 
  For every $\theta\in[0,1/2]$, let %\footnote{We will also omit the dependency of $\lambda$ on $\theta$ from now on.} 
  $\lambda \equiv \lambda(\theta) \in [0,1/4]$ be the unique value satisfying 
  \begin{align}
    \label{eq:thetasolution}
    \theta
    = \frac{1}{2} - \frac{\lambda\log\left(\frac{1 + \sqrt{1 - 4\lambda}}{1 - \sqrt{1 - 4\lambda}}\right)}{\sqrt{1 - 4\lambda}}
    = \frac{1}{2} - 2\lambda \sum_{n \geq 0} \frac{(1-4\lambda)^n}{2n+1}
    .
  \end{align}
 
  Then we define
 
  \begin{align}
    f(\theta)& =  - \theta \log(1-4\lambda) - (1-2\theta) \log(\lambda) + 2(\log(2)-1)\theta,\label{eq:f_solution}\\
    {J}(\theta)&=\sqrt{\frac{2}{\lambda(1-2\theta-4\lambda+4\theta\lambda)}}.\label{eq:J_solution}
  \end{align}
\end{theorem}

\begin{figure}[t]
  \centering
  \newcommand{\myplotscale}{0.3\linewidth}
  \includegraphics[width=\myplotscale]{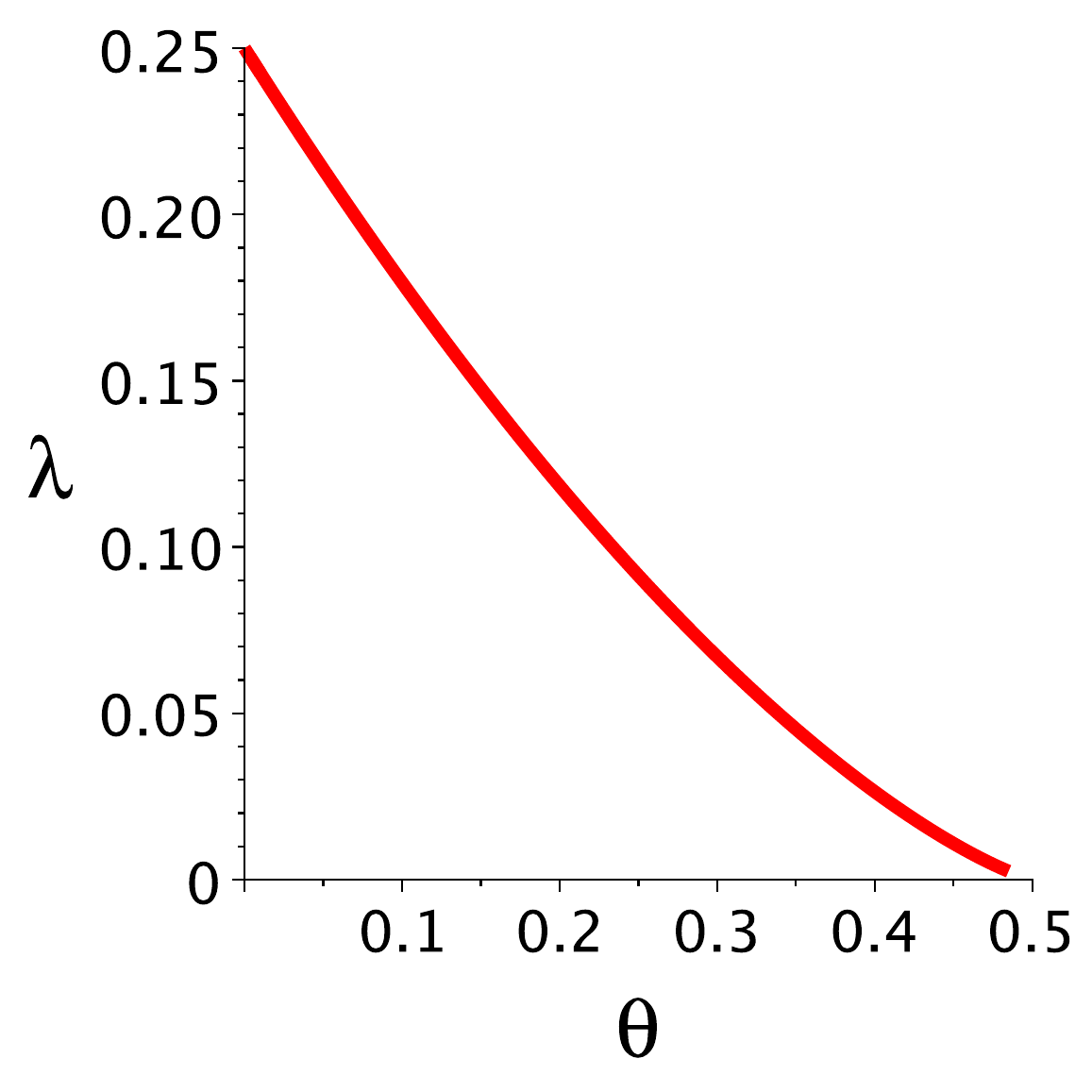}
  \quad
  \includegraphics[width=\myplotscale]{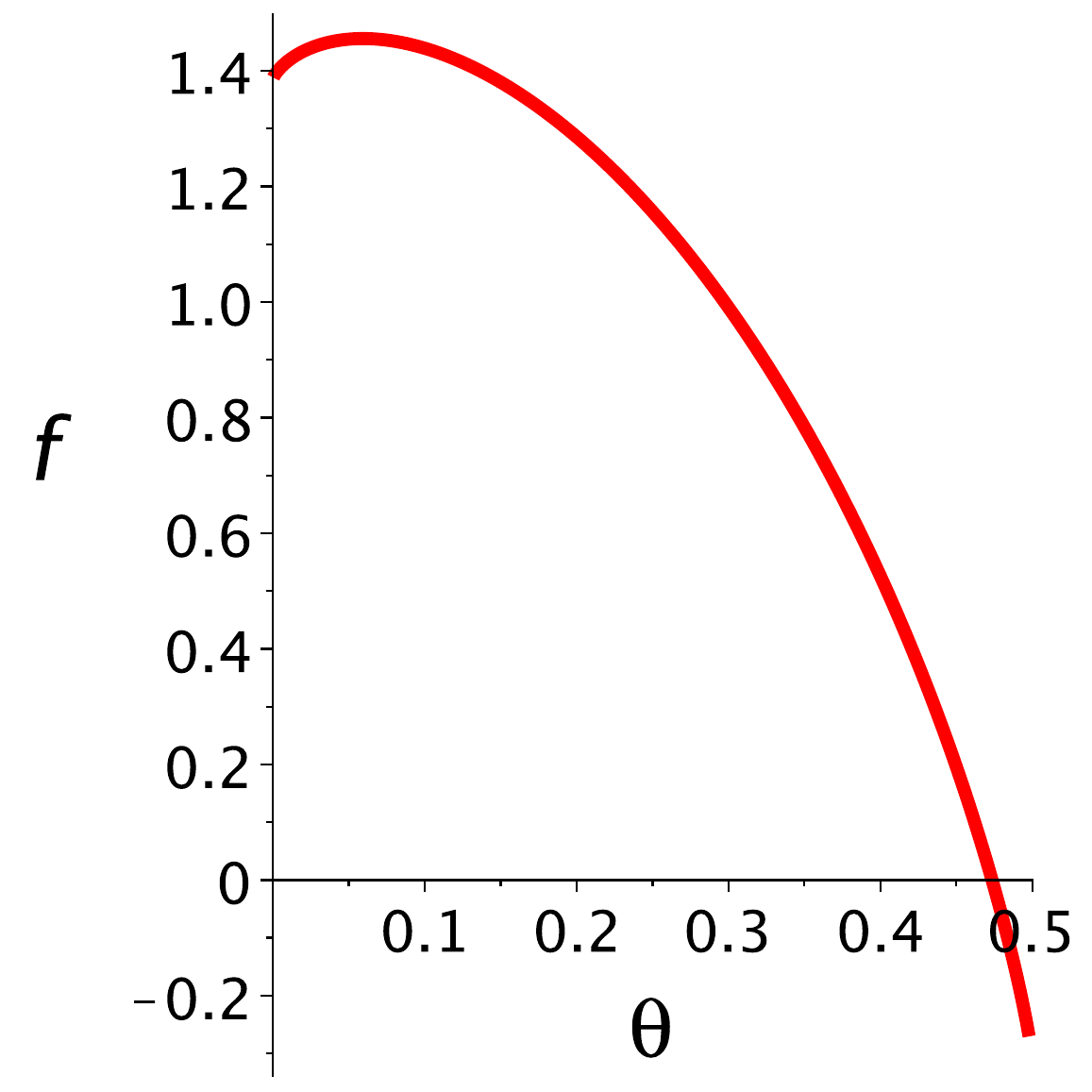}
  \quad
  \includegraphics[width=\myplotscale]{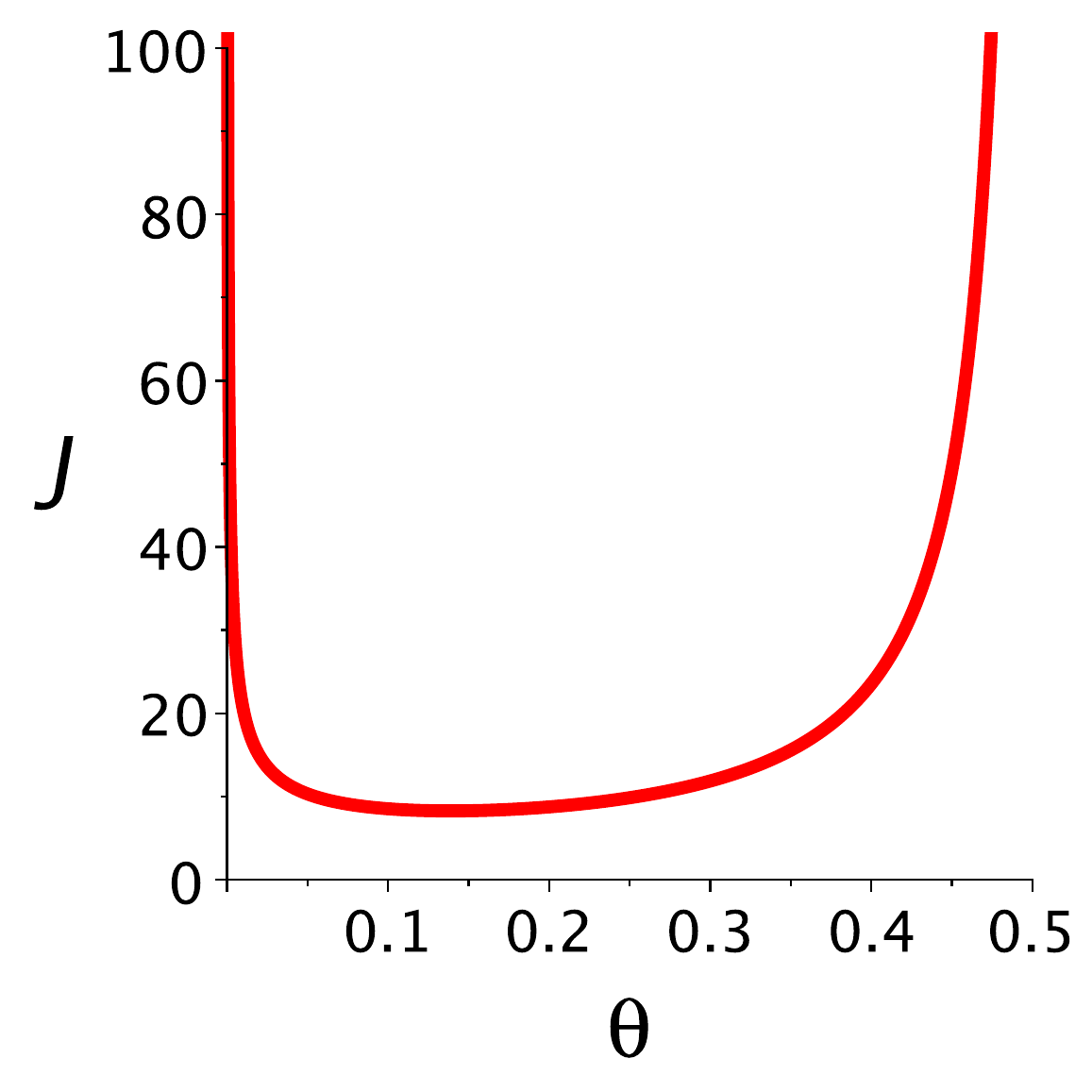}
  \caption{Plots of $\lambda$, $f$, and $J$ with respect to $\theta$. 
    Note that for $\theta \to 0^+$ it holds that $\lambda$ and $f$ tend to $1/4$ and $2\log(2) \approx 1.39$, %respectively,
    while for $\theta \to 1/2^-$ they tend to $0$ and $\log(2)-1 \approx -0.3$, respectively.
    The function $J$ tends to $+\infty$ for both these limits.
  }
  \label{fig:lambda-f-J}
\end{figure}

\begin{remark}[Number of vertices]\label{rem_vertices}
  For unicellular maps, Euler's formula states that the number of vertices $v \equiv v_n$ is given by $v = n + 1 - 2g$. 
  Therefore, the range of our main theorem is equivalent to $\frac{v}{\log n} \to \infty$.
\end{remark}

\begin{remark}[Genus zero]\label{rem_g0}
  %% For theta close to 0, lambda = 1/4 - 3/4*theta + 9/20*theta^2 + 27/700*theta^3 + ...
  For $g=0$, one needs to take the continuous limit $g\rightarrow 0$ in the right-hand side of~\eqref{eq_main_asympto}. This gives
  $\E(n,0)\sim \frac{4^n}{n^{3/2}\sqrt{\pi}},$
  which is consistently the asymptotics of the Catalan number $\frac{1}{n+1}\binom{2n}{n}$ that is the closed-form solution of $E(n,0)$ (unicellular maps of genus $0$). 
\end{remark}

\begin{remark}[Infinite genus]\label{rem:g-infinite}
  If we suppose further that $g_n \to \infty$, then using Stirling's approximation, we may simplify \eqref{eq_main_asympto} to:
  \[
    \E(n,g) \sim \frac{1}{2\sqrt{2}\pi} n^{2g-2} e^{n f(\frac{g}{n})} {J}\left(\frac{g}{n}\right).
  \]
  In fact, we started by guessing the simplified asymptotic form above, then adapted it to the case where $g_n$ is bounded. See \Cref{sec:guessing} for more details.
\end{remark}

It is important to note that \Cref{thm_main} has already been proven for $g\in[\eps n,(1/2-\eps)n]$ \cite{linear-genus} and $g=O(n^{1/3})$ \cite[Corollary~4]{meso-genus} using the combinatorics of the model. In fact, there exists a simple formula for the generating series of (a rescaling of) the $E(n,g)$ from which \Cref{thm_main} can be proven using classical saddle point analysis. The existence of a simple formula is very particular to the fact that the Harer Zagier recursion is equivalent to an order one linear differential equation for the generating series of a well chosen rescaling of the $E(n,g)$ (see~\Cref{sec:midv} for more details).

Therefore, in this paper, we present not only asymptotics for unicellular maps, but a \emph{general method} that should yield bivariate asymptotics results for many recurrences like~\eqref{eq:recurrence}, without having to rely on the combinatorics of the model, or a simple closed equation for the generating series, which is very rare. In fact our method is based on a ``guess and check'' approach, and in \cref{thm_walk_general}, we provide a general result for the checking part, using a random walk approach. See \cref{sec:discussion} for a more detailed discussion.

While \cref{thm_main} deals with the regime $\largev$, our methods also allow us to obtain asymptotics for the regime \smallv.

\begin{theorem}\label{thm_smallv}
  Given a sequence\footnote{We will use the notation $g \equiv g_n$ to denote the omission of the dependency of $g$ on $n$. 
    % We will omit the dependency of $g$ on $n$ from now on.
  } $g \equiv g_n$ such that $\frac{n-2g}{\log n}\rightarrow 0$ as $n\to\infty$, the following asymptotics hold:
  \begin{equation}\label{eq_asympto_smallv}
    \E(n,g)\sim\frac{1}{\sqrt{\pi}} n^{-3/2} 2^n \frac{n!}{(n-2g)!}\log(n)^{n-2g}.
  \end{equation}
\end{theorem}

For the sake of completeness, we also include asymptotics in the regime $\frac{n-2g}{\log n}\to c$, although we do not use our approach to prove it, but more classical methods (the exact generating function and saddle point asymptotics).

\begin{theorem}\label{thm_logv}
  For every constant $c>0$ and sequence $g \equiv g_n$ such that $\frac{n-2g}{\log n}\rightarrow c$ as $n\to\infty$, the following asymptotics hold:
  \begin{equation}\label{eq_asympto_midv}
    \E(n,g)\sim \frac{2^{c}}{c\Gamma(c)}\frac{1}{\sqrt{\pi}} n^{-3/2} 2^n \frac{n!}{(n-2g)!}\log(n)^{n-2g}.
  \end{equation}
\end{theorem}

\section{Informal proof ideas: asymptotic guess and check}\label{sec:guess}

In this section, we present our methodology that we call \emph{asymptotic guess and check}, and how it is used to obtain our main result \Cref{thm_main}. It consists of several ingredients.

\paragraph{\textbf{Asymptotic guessing.}} Suppose that we have a linear bivariate recurrence on some quantity $E(n, g)$ with explicitly known boundary conditions. In our case, the recurrence \eqref{eq:recurrence} satisfies such conditions, with $E(n, g)$ known on the boundaries $g = 0$ and $n = 2g$. We start by guessing an explicit formula $\Omega(n, g)$ that satisfies the following conditions.
\begin{enumerate}
\item \textbf{Asymptotic initial condition}: $\Omega(n, g)$ agrees with known $E(n, g)$ \emph{asymptotically} on some boundary. In our case, for $n \to \infty$, we should have
  \[     
    \Omega(n,0) \sim \E(n,0).
  \]
\item \textbf{Asymptotic recurrence}: $\Omega(n, g)$ should follow the same recurrence as $E(n, g)$ \emph{asymptotically}, but here with a well-chosen and stricter notion of ``asymptotically'' than usual to better control the errors in later steps. In our case, we need the following to hold with a special notion of $\approx$:
  \begin{align*}
    (n+1) \Omega(n,g) \approx 2 (2n-1) \Omega(n-1, g) + (n-1) (2n-1) (2n-3) \Omega(n-2, g-1).
  \end{align*}
\end{enumerate}
Although such guessing may use prior knowledge about $E(n, g)$, it is not necessary. When not using prior knowledge, which is our case here, it can usually be carried out as follows:
\begin{enumerate}
\item By numerical computation, we first take $g = \theta n$ for several fixed values of $\theta$, and empirically determine the asymptotic form of $E(n, \theta n)$ as $n\to\infty$, where some terms in this asymptotic form are functions of $\theta$ that we call \emph{parametric functions}.
\item We then substitute such an asymptotic form into the discrete recurrence and develop it around some point $(n, \theta n)$ as $n\to\infty$. If our guessed form is correct, the terms of highest orders should involve only $\theta$, giving ordinary differential equations of the parametric functions. Solving these equations gives us our guess for the asymptotic behaviour of $E(n, g)$, whose precision depends on the number of terms we consider in the development of the discrete recurrence. Such a guess may need further refinement on the boundary by multiplying some extra terms to meet the precision requirement in the following step of asymptotic checking.
\end{enumerate}
In \Cref{sec:guessing}, we take these steps to obtain our guess for $\Omega(n, g)$.

\paragraph{\textbf{Asymptotic checking via random walks.}} As we want to show that $\Omega(n, g) \sim E(n, g)$ asymptotically, it is natural to consider the quotient $Q(n, g) = E(n, g) / \Omega(n, g)$ and try to show that $Q(n, g) \to 1$ for large $n$. The bivariate recurrence can be rewritten as a recurrence for $Q(n, g)$. In our case, it takes the form
\[
  Q(n, g) = \alpha(n, g) Q(n-1, g) + \beta(n, g) Q(n-2, g-1).
\]
We may use the (normalised) coefficients to define a random walk from $(n, g)$ to $(0, 0)$, and we denote by $(n', g')$ a generic point on it. In our case, the steps of the random walk are $(-1, 0)$ and $(-2, -1)$. To use such a random walk to show that $Q(n, g) \to 1$ when $n \to \infty$, we need the following conditions.
\begin{enumerate}
\item \textbf{Negligible cumulative error}: The random walk does not conserve the expectation of $Q(n', g')$ for points $(n', g')$ visited, as it introduces a multiplicative error at each step due to normalisation, because the coefficients do not sum to $1$ in general. However, if the asymptotic recurrence holds for $\Omega(n, g)$, with a sufficiently precise notion of $\approx$ (see \Cref{assum_alpha_beta} in our case), we may still control the error to be $o(1)$, and the random walk approximately conserves the expectation of $Q(n', g')$.
\item \textbf{Anchoring on ``good'' boundaries}: To anchor the value of $Q(n, g)$, we terminate the random walk at boundaries on which the exact value of $Q(n', g')$ can be computed. We will have a ``good'' boundary, where $Q(n', g') \to 1$ for large $n'$ when $(n', g')$ on this boundary, and a ``bad'' boundary, where $Q(n', g')$ is bounded by a constant (see \Cref{assum_bd_values}). This should be possible if our guess from the asymptotic guessing part has enough precision. It thus suffices to show that the random walk terminates on a ``good'' boundary almost surely. To this end, we use another quantity $s(n, g)$ whose expectation is also conserved approximately along the random walk, while behaving differently on ``good'' and ``bad'' boundaries (see \Cref{assum_behavior_s}). We then conclude with the fact that $Q(n', g')$ is bounded above by a constant on a ``bad'' boundary and $s(n, g)$ is approximately $o(1)$ in the considered region, but bounded below by a positive constant on the ``bad'' boundary.
\end{enumerate}
If the conditions above are satisfied, it is clear that the random walk almost surely ends on a ``good'' boundary, on which the expectation of $Q(n, g)$, which is approximately conserved, tends to $1$ when $n \to \infty$, showing that $\Omega(n, g) \sim E(n, g)$. We note that the conditions above may only hold within a certain region. It is thus necessary to perform similar analysis on every region we want to cover.

We remark that the asymptotic guessing and checking are independent of each other. Therefore, even for more general recurrences, as long as we can somehow guess an asymptotic form of the quantities that satisfies the conditions in the asymptotic checking part, we may still apply the same method to prove the validity of the guessed asymptotic form.

\section{Heuristic guessing} \label{sec:guessing}

We discuss now how we guess a bivariate asymptotic form $\Omega(n,g)$ for $E(n,g)$, as our proof will then revolve around the analysis of the ratios $E(n,g)/\Omega(n,g)$. 
First, we use the recurrence to compute the exact values $E(n,g)$ for $n\leq 1000$. Then, using standard, univariate empirical analysis (see \cite{guttmannser} and references therein), we derive precise asymptotic estimates for subsequences, e.g.,
\[\E(3g,g)\sim c_{1}g^{2g-2}\mu_{1}^{g}~~~~~\text{and}~~~~~\E(4g,g)\sim c_{2}g^{2g-2}\mu_{2}^{g},\]
with constants $c_{1}\approx 0.042124$ and $c_{2}\approx0.033183$, and growth rates $\mu_{1}\approx 117.923$ and $\mu_{2}\approx 1633.26$. Analysing $\E(n,\theta n)$ for different fixed, rational values of $\theta$ yields similar results, which leads us to predict that
\[\E(n,\theta n)\sim n^{2\theta n-2}e^{nf(\theta)}{J}\left(\theta\right),\]
for some functions $f$ and $J$. 
To determine the function $f$, we substitute the approximate expression above for $\E$ into \eqref{eq:recurrence} with $g=\theta n$, divide by the left-hand side, and then take the limit $n\to\infty$. This yields the differential equation
\[1=4e^{-2\theta-f(\theta)+\theta f'(\theta)}+4e^{-4\theta-2f(\theta)+2\theta f'(\theta)-f'(\theta)},\]
which can be solved exactly to give our expression for $f(\theta)$. To solve the differential equation, it helps to set 
\begin{equation}
    \label{eq:lambdadef}
    \lambda(\theta) = e^{-2\theta-f(\theta)+\theta f'(\theta)},
\end{equation}
as then we have
\begin{equation}
    \label{eq:nicelambdaf}
    1=4\lambda(\theta)+4\lambda(\theta)^{2}e^{-f'(\theta)}.
\end{equation}
On the one hand, combining \eqref{eq:lambdadef} and \eqref{eq:nicelambdaf} yields the expression \eqref{eq:f_solution} for $f$ in terms of $\lambda$ and $\theta$.
On the other hand, rearranging and differentiating \eqref{eq:lambdadef} and \eqref{eq:nicelambdaf} yields
\[
    \log(\lambda(\theta))'=-2+\theta f''(\theta) 
    \qquad \text{ and } \qquad 
    \log\left(\frac{1-4\lambda(\theta)}{4\lambda(\theta)^{2}}\right)'=-f''(\theta).
\]
Combining these equations to cancel the term $f''(\theta)$ yields 
\[\frac{\lambda'(\theta)}{\lambda(\theta)}=-2+\theta\left(\frac{4\lambda'(\theta)}{1-4\lambda(\theta)}+2\frac{\lambda'(\theta)}{\lambda(\theta)}\right).\]
This simplifies further if we consider the inverse function $\theta(\lambda)$ of $\lambda(\theta)$ (in some region where it is well-defined). In particular, we get the following linear differential equation:
\[2\theta'(\lambda)=-\frac{1}{\lambda}+\left(\frac{4}{1-4\lambda}+\frac{2}{\lambda}\right)\theta(\lambda).\]
Solving this yields the explicit solution~\eqref{eq:thetasolution} up to a constant, which we fix by setting $\theta(0)=1/2$ to match with the known genus $0$ asymptotics of $E(n,g)$.
%In order to determine the initial value of $\theta$ for $\lambda=0$, we use \eqref{eq:f_solution}. 
%Observe that $f(\theta)$ can only be bounded around $\lambda=0$ if $\theta=1/2$, hence $\theta(0)=1/2$.

To estimate $J(\theta)$ we again substitute our approximate expression for $\E$ into \eqref{eq:recurrence}, but this time we analyse the $n^{-1}$-term in the resulting equation. Doing so yields the equation
\[\frac{J'(\theta)}{J(\theta)}=\frac{4\kappa\lambda(\theta)-2\kappa+6\lambda(\theta)-4}{4\theta\lambda(\theta)-2\theta-4\lambda(\theta)+1}-\frac{4\theta\lambda(\theta)}{(4\theta\lambda(\theta)-2\theta-4\lambda(\theta)+1)^2}+f'(\theta)+2,\]
which can be solved up to a constant term.

Finally, we multiply the expression by $\frac{1}{2\sqrt{\pi}} \frac{\sqrt{g}(g/e)^g}{g!}$ so that it holds for fixed $g$, which gives our estimate \eqref{eq_main_asympto}. One can check that this is precisely consistent with our asymptotic estimates of $\E(3g,g)$ and $\E(4g,g)$.

\section{The random walk}\label{sec:RW}

The proofs of this section apply to the Harer-Zagier recursion, but the arguments used do not rely on the specifics of its coefficients, which hints at the fact that a general argument might exist. And indeed, we prove this in~\cref{sec:RW_general}.

\subsection{Setup of the walk}

Given an explicit family $\left(\Omega(n,g)\right)_{n,g \in I}$ of positive numbers, we introduce
\begin{align}
Q(n,g)&:=\frac{\E(n,g)}{\Omega(n,g)}, \label{eq_Q_def}\\
    \alpha(n,g) &:= \frac{2(2n-1)}{n+1} \cdot \frac{\Omega(n-1,g)}{\Omega(n,g)}, \label{eq_alpha_def}\\
    \beta(n,g) &:= \frac{(n-1)(2n-1)(2n-3)}{n+1} \cdot \frac{\Omega(n-2,g-1)}{\Omega(n,g)}. \label{eq_beta_def}
\end{align}
Then the Harer--Zagier recursion~\eqref{eq:recurrence} can be rewritten in the following compact form
\begin{equation}\label{eq_alpha_beta_Q}
    Q(n,g)=\alpha(n,g)Q(n-1,g)+\beta(n,g)Q(n-2,g-1).
\end{equation}

Additionally, our walk will be confined to a certain domain defined by two boundaries. 
These are described by two explicit non-decreasing sequences $\goodG$ and $\badG$. 
We require that $\goodG-\badG$ never changes sign, which means that these boundaries never cross.
Depending on which one dominates, we additionally define the following technical quantities $\interv$, $\badGG$, $\goodGG$, $s(n,g)$, which will lead to an intuitive description of our method.

\begin{center}
\begin{tabular}{@{} L L C C C@{}}
\toprule
\text{Description} & \text{Notation} & & \text{If }\goodG<\badG & \text{If }\goodG>\badG \\
\midrule
\text{Inner integer points} & \interv & =  & (\goodG,\badG)\cap\mathbb Z & (\badG,\goodG)\cap\mathbb Z \\
\text{Shifted bad boundary} & \badGG & =  & \badG-1                    & \badG+1 \\
\text{Shifted good boundary} & \goodGG & = & \goodG+1                   & \goodG-1 \\
 \text{Omega-quotient} & s(n,g) & =  & \dfrac{\Omega(n,g-1)}{\Omega(n,g)} & \dfrac{\Omega(n,g+1)}{\Omega(n,g)} \\
\bottomrule
\end{tabular}
\end{center}

The boundaries naturally define a cone $C$ given by the integer points between the boundaries:
$C = \{ (n,g) : g \in \interv \}.$
The shifted boundaries will be needed in the proofs and are crucial in \Cref{assum_behavior_s} below. Note that in general we cannot control the behaviour on the boundary, but on the shifted boundary we can. 
The Omega-quotient $s(n,g)$ compares $\Omega(n,g)$ for consecutive values of $g$ and fixed $n$. It will allow us to prove that the good boundary $\goodG$ attracts the random walk in most cases; see %\Cref{lem_E(S)} and 
\Cref{prop_behavior_RW} below. 
Now we define a random walk which starts somewhere in $C$ and then moves towards the origin $(0,0)$.

\paragraph{The random walk}
Given $n,g$ we can define  the random walk $(N_k,G_k)_{k\geq 0}$ as follows:

\begin{itemize}
    \item \emph{Starting point:} $(N_0,G_0)=(n,g)$
    \item \emph{Stopping time $\tau$:} The walk is stopped as soon as $G_k=\goodGNK$ or $G_k=\badGNK$. We call $\tau=\tau(n,g)$ the stopping time.
    
    \item \emph{Transition:} At step $0\leq k<\tau$, we have the following transitions\footnote{The transitions have independent sources of randomness for each step.}:
    \begin{align}
    \label{eq:probarw}
    \begin{aligned}
    &(N_{k+1},G_{k+1})= (N_{k}-1,G_{k}) &\text{with probability} &\quad \frac{\alpha(N_k,G_k)}{\alpha(N_k,G_k)+\beta(N_k,G_k)},\\
    &(N_{k+1},G_{k+1})= (N_{k}-2,G_{k}-1)&\text{with probability} &\quad \frac{\beta(N_k,G_k)}{\alpha(N_k,G_k)+\beta(N_k,G_k)}.
    \end{aligned}
    \end{align}
\end{itemize}

\subsection{Main result}
Our main result relies three main conditions that we state below. This modular presentation will allow the result to be reused for similar problems. The key difficulty to proving these conditions will be finding a good explicit family $\Omega(n,g)$.
The first condition describes the condition that our weights $\alpha$ and $\beta$ are good approximations of probabilities, where ``good'' means that there is a summable error term. 
\begin{condition}[Good asymptotic probabilities]\label{assum_alpha_beta}
As $n\to\infty$, uniformly in \gstrict, and also for $g=\goodG$, 
\[\alpha(n,g)+\beta(n,g)=1+O\left((n-g)^{-1}\log^{-2}(n-g)\right).\]
\end{condition}

The second condition describes the asymptotic behaviour at the boundaries. 
The behaviour here is the reason for the naming conventions using the words ``good'' and ``bad''. 
On the good boundary, our approximation is perfect, while on the bad boundary it does not hold. 
\begin{condition}[Boundary values]\label{assum_bd_values}
    As $\nto$  we have
    \[Q(n,\goodG)\to 1 \quad \text{and}\quad Q(n,\badG)\to 0\]
\end{condition}

The third and final condition describes the behaviour of $s$. 
\begin{condition}[Behaviour of $s$]\label{assum_behavior_s}
The numbers $s(n,g)$ satisfy:
\begin{itemize}
    \item $s(n,\goodGG)>0$ for all $n\geq 0$;
    \item there exists a constant $c>0$ such that $s(n,\badGG)>c$ for all $n\geq 0$.
\end{itemize}
\end{condition}

If these three conditions are satisfied, then the following theorem holds. 
\begin{theorem}\label{thm_RW_to_asympto}
Suppose that \Cref{assum_alpha_beta,assum_bd_values,assum_behavior_s} are satisfied, and also that the sequence $g\equiv g_n$ satisfies $s(n,g)\to 0$ as $\nto$. Then for $\nto$ we have
\[\E(n,g)\sim\Omega(n,g).\]
\end{theorem}

In the following sections, we will apply this to the regimes $n-2g=\omega(\log n)$ and $n-2g=o(\log n)$ separately, with different definitions for $\Omega(n,g)$, $\goodG$, and $\badG$.

\subsection{Proof of \Cref{thm_RW_to_asympto}}

Recall that the random walk starts at $(N_0,G_0)=(n,g)$. Set 
\begin{align*}
    S_k&=s(N_k,G_k), &
    Q_k&=Q(N_k,G_k), & \text{and} &&
    M_k&=N_k-G_k. 
\end{align*}
Notice that $M_k$ is a deterministic quantity, and in particular: $M_k=M_0-k$. 
It has the following geometric interpretation:
After $k$ steps, the random walk is on the diagonal $y = x + M_k$.
This follows directly from the steps $(-1,0)$ and $(-2,-1)$ from \eqref{eq:probarw}.
This fact, along with the recurrence~\eqref{eq_alpha_beta_Q} and \Cref{assum_alpha_beta} shows that $Q_k$ does not change much in expected value (and the same goes for $S_k$ thanks to \Cref{assum_behavior_s}). To make this precise, we define the error terms
\begin{align*}
    r^{+}(k)&:=\prod_{j=k+1}^{\infty}\max\left(\left\{ \alpha(n,g)+\beta(n,g)  \mid n-g=j, \, 0<g<(n-1)/2 \right\}\right),\\
    r^{-}(k)&:=\prod_{j=k+1}^{\infty}\min(\left\{ \alpha(n,g)+\beta(n,g) \mid n-g=j, \, 0<g<(n-1)/2 \right\}).
\end{align*}
Since each $\Omega(n,g)$ is positive, by \eqref{eq_alpha_def} and \eqref{eq_beta_def}, the error terms $r^{-}(k)$ and $r^{+}(k)$ are also positive. Moreover, by \Cref{assum_alpha_beta}, they satisfy 
\begin{align}
    \label{eq:rpm_limit_bhr}
%    r^{-}(k),r^{+}(k) &\in (0,\infty) &\text{ and }&& 
    r^{-}(k),r^{+}(k) &\to 1 \quad \text{ as $k\to\infty$}.
\end{align}
These quantities allow us to bound the expected value of $Q_\tau$, i.e., the value when the random walk hits one of the boundaries, compared to its initial value $Q_0$.
The key observation is that most of its initial value $Q_0$ is conserved in $Q_\tau$, as shown in the following result.

\begin{proposition}[Conserved quantity]
\label{prop_E(Q)}
We have 
    \[
        \mathbb{E}\left( \frac{r^{-}(M_{\tau})}{r^{-}(M_{0})}Q_\tau \right) \leq Q_0 \leq \mathbb{E}\left( \frac{r^{+}(M_{\tau})}{r^{+}(M_{0})}Q_\tau \right).
    \]
\end{proposition}

\begin{proof}
 In order to compute the expectation we condition on the the specific ending point on the boundary $\partial C = \{ (n',g') : g'=\good{n'} \text{ or } g'=\bad{n'}\}$ of our cone $C$. 
    We get
    \begin{align}
        \mathbb{E}(r^{-}(M_{\tau})Q_\tau) 
            &= \sum_{(n',g') \in \partial C} \mathbb{E}\left(r^{-}(M_{\tau})Q_\tau \mid (N_{\tau},G_{\tau}) = (n',g')\right) \mathbb{P}\left( (N_{\tau},G_{\tau}) = (n',g') \right) \notag \\
            &= \sum_{(n',g') \in \partial C} r^{-}(n'-g') Q(n',g') \mathbb{P}\left( (N_{\tau},G_{\tau}) = (n',g') \right) \notag \\
            &= \sum_{k \geq 2} \sum_{\substack{(n',g') \in \partial C\\n'-g'=k}} r^{-}(k) Q(n',g') \mathbb{P}\left( (N_{\tau},G_{\tau}) = (n',g') \right).\label{€q:Erminus_Qtau_bound}
    \end{align}
  
    We will now bound the probability that the random walk starting at $(n,g)$ ends at $(n',g')$, which appears in the expressions above; see~\eqref{eq:probarw}.
    Observe that each such path $\omega$ touches each diagonal $n''-g''=j$ at most once. 
    We denote by $\frac{\gamma_{j}(\omega)}{\alpha_{j}(\omega)+\beta_{j}(\omega)}$ the probability of the step of $\omega$ that starts at the diagonal $n''-g''=j$ of $\omega$, where $\gamma_j(\omega)$ is equal to either $\alpha_j(\omega)$ or $\beta_j(\omega)$, depending on whether the step is $(-1,0)$ or $(-2,-1)$, respectively.
    Therefore, we get by the definition of $r^{-}(k)$
    \begin{align*}
        r^{-}(k) \mathbb{P}\left( (N_{\tau},G_{\tau}) = (n',g') \right) 
            &= \sum_{\omega: (n,g) \to (n',g')} \prod_{j \geq k+1}^{n-g} \frac{\gamma_{j}(\omega)}{\alpha_{j}(\omega)+\beta_{j}(\omega)} r^{-}(k)\\
            &\leq r^{-}(n-g) \sum_{\omega: (n,g) \to (n',g')} \prod_{j \geq k+1}^{n-g} \gamma_{j}(\omega),
    \end{align*}
    where the sum runs over all paths $\omega$ from $(n,g)$ to $(n',g')$.  
    Note that in order to make the product above well-defined we set $\frac{\gamma_j(\omega)}{\alpha_j(\omega)+\beta_j(\omega)}=1$ if the walk does not cross the diagonal $n''-g''=j$.
    Using this bound in~\eqref{€q:Erminus_Qtau_bound} we get
    \begin{align*}
        \mathbb{E}(r^{-}(M_{\tau})Q_\tau) 
            &\leq r^{-}(n-g) \sum_{k \geq 2} \sum_{\substack{(n',g') \in \partial C\\n'-g'=k}} Q(n',g') \sum_{\omega: (n,g) \to (n',g')} \prod_{j \geq k+1}^{n-g} \gamma_{j}(\omega) \\
            &= r^{-}(M_0) Q(n,g),
    \end{align*}
    since the complicated sum is nothing else but the sum over all paths of the recurrence~\eqref{eq_alpha_beta_Q} from $(n,g)$ to the diagonal $\partial C$. %; see also~\eqref{eq:sum_walks}.
    
    The upper bound follows analogously.
\end{proof}

We also need a similar (but coarser) property on $S_k$:

\begin{lemma}[Conserved quantity, bis]\label{lem_E(S)}
Let $S_k:=s(N_k,G_k)$, then
\[\mathbb{E}(S_{\tau-1})=\bigO{S_0}.\]
\end{lemma}

\begin{proof}

In the following, we use $\pm$ to mean $-$ if $\goodG<\badG$ and $+$ if $\goodG>\badG$, so that $s(n,g)$ is defined by
\[s(n,g)=\frac{\Omega(n,g\pm 1)}{\Omega(n,g)}.\] From the Definitions~\eqref{eq_alpha_def} and \eqref{eq_beta_def} for $\alpha(n,g)$ and $\beta(n,g)$, respectively, we immediately have the exact equality
\[\frac{\alpha(n,g)s(n-1,g)+\beta(n,g)s(n-2,g-1)}{s(n,g)}=\alpha(n,g\pm1)+\beta(n,g\pm1).\]
Here it is important that the first factors of $\alpha(n,g)$ and $\beta(n,g)$ in~\eqref{eq_alpha_def} and \eqref{eq_beta_def} are independent of~$g$. 
This equation can now be rewritten as
\begin{align*}
s(n,g) &=\frac{\alpha(n,g)}{\alpha(n,g\pm1)+\beta(n,g\pm1)}s(n-1,g)  + \frac{\beta(n,g)}{\alpha(n,g\pm1)+\beta(n,g\pm1)}s(n-2,g-1)\\
&=\left(1+O\left((n-g)^{-1}\log^{-2}(n-g)\right)\right)\left(\alpha(n,g)s(n-1,g)  + \beta(n,g)s(n-2,g-1)\right),
\end{align*}
where the second equality follows from~\cref{assum_alpha_beta}.
Therefore, recalling that $N_k-G_k$ is deterministic and equal to $N_0-G_0-k$, we get
\[\mathbb{E}(S_k)=\left(1+O\left((N_0-G_0-k)^{-1}\log^{-2}(N_0-G_0-k)\right)\right)\mathbb{E}(S_{k-1}),\]
and thus $\mathbb{E}(S_k)=O(S_0)$ for all $k$ since $\frac{1}{x\log(x)^2}$ is summable. This finishes the proof.
\end{proof}

\begin{proposition}[Behaviour of the random walk]\label{prop_behavior_RW}
    The sequence $g\equiv g_n$ of \Cref{thm_RW_to_asympto} satisfies the following: for all fixed $L>0$, the following holds as $\nto$
    \[\P(G_\tau=\goodGNT\text{ and } N_\tau>L)\to 1.\]
\end{proposition}

\begin{proof}
    On the one hand, by \Cref{lem_E(S)} one has, as $\nto$ 
    \[
        \P(G_\tau=\bado{N_\tau}) \mathbb{E}(S_{\tau-1} \mid G_\tau=\bado{N_\tau}) \leq \mathbb{E}(S_{\tau-1})=\bigO{S_0}=o(1),
    \]
    where we use the condition $S_{0}\to 0$ of \Cref{thm_RW_to_asympto}. But due to \Cref{assum_behavior_s}, we have the following
    \[\mathbb{E}(S_{\tau-1} \mid G_\tau=\bado{N_\tau}) = \mathbb{E}(s(N_{\tau-1},G_{\tau-1}) \mid G_\tau=\bado{N_\tau})>c>0.\]
    Therefore one has 
    \begin{equation}\label{eq_XX}
        \P(G_\tau=\bado{N_\tau})=o(1).
    \end{equation}

Fix now a constant $L$. Using again \Cref{lem_E(S)} and $S_{0}\to 0$, we have
\[
    \P(N_\tau\leq L\text{ and }G_\tau=\good{N_\tau})\mathbb{E}(S_{\tau-1} \mid N_\tau\leq L\text{ and }G_\tau=\good{N_\tau})\leq \mathbb{E}(S_{\tau-1})=o(1).
\]
Due to the other condition of~\Cref{assum_behavior_s} it holds that 
\[\mathbb{E}(S_{\tau-1} \mid N_\tau\leq L\text{ and }G_\tau=\good{N_\tau})\geq \min_{1\leq j\leq L+1}s(j,\goodo{j})>0.\] 
Therefore we have
\begin{equation}\label{eq_YY}
   \P(N_\tau\leq L\text{ and }G_\tau=\good{N_\tau})=o(1).
\end{equation}
Finally, \eqref{eq_XX} and \eqref{eq_YY} imply the result.
\end{proof}

We are finally ready to prove the main result of this section.
\begin{proof}[Proof of \Cref{thm_RW_to_asympto}]

    By~\Cref{assum_bd_values}, we can bound $Q(n,\goodG)$ and $Q(n,\badG)$ by a constant $C$ for all $n$, and therefore $Q_\tau\leq C$ deterministically. Hence, by \Cref{prop_behavior_RW}, it holds that
    \[\mathbb{E}(r^{\pm}(M_{\tau})Q_\tau)=\mathbb{E}(r^{\pm}(M_{\tau})Q_\tau \mid G_\tau=\goodGNT\text{ and }N_{\tau}>L)+o(1).\]
    % \[\mathbb{P}(N_{\tau}=2G_\tau+1)\mathbb{E}(Q_\tau \mid N_{\tau}=2G_\tau+1)=o(1)\quad\text{and} \quad \mathbb{P}(N_{\tau}<\infty)\mathbb{E}(Q_\tau \mid N_{\tau}=\infty)=o(1), \]
    Therefore, by \Cref{prop_E(Q)} and \eqref{eq:rpm_limit_bhr}, we have as $n\to\infty$
    \[Q(n,g)=Q_0\leq\frac{1}{r^{+}(M_0)}\mathbb{E}\left(r^{+}(M_{\tau}) Q_\tau \right)\leq \max_{N>L}(r^{+}(N)Q(N,\good{N}))+o(1).\]
    %First, note that $r^{+}(M_0) = 1 + o(1)$ for $n \to \infty$ due to \eqref{eq:rpm_limit_bhr} since we have $M_0=n-g$.
    By~\Cref{assum_bd_values} we have $r^{+}(N)Q(N,\good{N})\to 1$ as $N\to\infty$. Hence, since the equation above holds for any fixed $L$, we have $Q(n,g)\leq 1+o(1)$.
    
    Using the same approach, but replacing $r^{+}$ with $r^{-}$ and reversing inequalities, we get the lower bound $Q(n,g)\geq 1+o(1)$. Combining these facts yields $Q(n,g)= 1+o(1)$, which completes the proof.
    
    %Using the same approach we get a similar lower bound on $Q(n,g)$ using $r^{-}(N)Q(N,\good{N})$. Combining these facts yields $Q(n,g)= 1+o(1)$, which completes the proof.
\end{proof}

\section{The regime $n-2g=\omega(\log n)$}\label{sec:largev}
% \subsection{The numbers $\Omega(n,g)$}

Take $\goodG=0$ and $\badG=\lceil\frac{n-1}{2}\rceil$. Set \begin{align}\label{eq_def_Olarge}
    \Omega(n,g)&:=\frac{1}{2\sqrt{\pi}} \frac{\sqrt{g}(g/e)^g}{g!}n^{2g-2}e^{nf(\frac{g}{n})}{J}\left(\frac{g}{n}\right)K(n-2g)\color{red}{}\color{black} \qquad \text{ with } \\
    K(x)&:=\frac{\sqrt{2\pi} \, x^{x+1}}{e^{x} \, \Gamma(x+3/2)},
    % \rightarrow 1\quad\quad\text{as $x\to\infty$},
    \notag
\end{align}
where $f$ and $J$ are the same as in \Cref{thm_main}.
Note that this expression is a priori not well-defined for $g=0$ as $J(0)=\infty$.
However, as discussed in \Cref{rem_g0}, the asymptotic behaviour of the different terms involved allows us to define ``by continuity''
\[\Omega(n,0):=\frac{n^{-3/2}4^n}{\sqrt{\pi}}.\]

\begin{proposition}\label{prop_alpha_beta_largev}
   \Cref{assum_alpha_beta,assum_bd_values} hold.
\end{proposition}

\begin{proof}
    The proof of~\Cref{assum_alpha_beta} is one of the key results of this paper, however its proof is quite technical, hence we postpone it until \Cref{sec:alpha_beta_largev}.

Let us prove \Cref{assum_bd_values}. Note that if $g=\badG$, then $n=2g+1$. From the Harer--Zagier recursion~\eqref{eq:recurrence} we directly deduce
\begin{align*}
    E(n,0)&=\frac{1}{n+1}\binom{2n}{n} &
    & \text{ and } &
    E(2g+1,g)\sim 2\sqrt{2}\left(\frac{4g}{e}\right)^{2g}\log(g).
\end{align*}
On the other hand, we have $\Omega(n,0)=\frac{n^{-3/2}4^n}{\sqrt{\pi}}$, hence $Q(n,0)\to 1$ as $\nto$. And we also have
\begin{equation}\label{eq_bad_largev}
    \Omega(2g+1,g)\sim c\left(\frac{4g}{e}\right)^{2g}\log(g)^{3/2}
\end{equation}
for some constant $c>0$, 
whose proof we postpone to the appendix (see \cref{sec:proof_eq_bad_largev}).

This entails
\[Q(2g+1,g)\to 0.\]
\end{proof}
\begin{proposition}\label{prop_s_largev}
    \Cref{assum_behavior_s} holds and if $g\equiv g_n$ is such that $\smallv$, then $s(n,g)\to 0$ as $\nto$.
\end{proposition}

Once again, due to the technicality of the proof, it is delayed until \Cref{sec:s_largev}.

\section{The regime $n-2g=o(\log n)$}\label{sec:smallv}
Take $\goodG=n/2$ and $\badG=\left\lceil\frac{n-\badV}{2}\right\rceil$. Set
\begin{align}\label{eq_def_Osmall}
    \Omega(n,g) &:=\frac{1}{\sqrt{\pi}} n^{-3/2} 2^n \frac{n!}{(n-2g)!}\log(n)^{n-2g}
    % \rightarrow 1\quad\quad\text{as $x\to\infty$},
    \notag
\end{align}

\begin{proposition}
   \Cref{assum_alpha_beta,assum_bd_values} hold.
\end{proposition}

\begin{proof}\label{prop_alpha_beta_smallv}
Let us first prove \Cref{assum_alpha_beta}.
Write $v=n-2g$. Then, recalling that here $v=o(\sqrt n)$
\begin{align}
    \alpha(n,g)&=\frac{2(2n-1)}{n+1}\frac{\bigpar{1-\frac{1}{n}}^{-3/2}v{\log\bigpar{n-1}^{v-1}}}{2n\log(n)^v}\\
    &=\bigpar{1+\bigO{\frac{1}{n}}}\frac{2v}{n\log(n)}\bigpar{1+\frac{\log\bigpar{1-\frac{1}{n}}}{\log(n)}}^{v-1}\\
    &=\frac{2v}{n\log(n)} +\bigO{\frac{1}{n^{3/2}}}
\end{align}
and 
\begin{align}
    \beta(n,g)&=\frac{(n-1)(2n-1)(2n-3)}{n+1}\frac{\bigpar{1-\frac{2}{n}}^{-3/2}\bigpar{1+\frac{\log\bigpar{1-\frac{2}{n}}}{\log(n)}}^{v}}{2n(n-1)}\\
    &=\bigpar{1+\bigO{\frac{1}{n^2}}}\bigpar{1+\frac{\log\bigpar{1-\frac{2}{n}}}{\log(n)}}^{v}\\
    &=1-\frac{2v}{n\log(n)} +\bigO{\frac{1}{n^{3/2}}}.
\end{align}

Now we turn to \Cref{assum_bd_values}.  From the Harer--Zagier recursion~\eqref{eq:recurrence} we directly deduce
\begin{align*}
    E(n,\goodG)
        % =\frac{(2n-1)!!}{n+1}
        =\frac{(2n)!}{2^{n}(n+1)!},
\end{align*}
which are for even $n$ the number of rooted maps of genus $\frac{n}{2}$ with one vertex and one face (and $n$ edges); see \OEIS{A035319}.
But $\Omega(n,\goodG)=\frac{2^nn!}{\sqrt{n^3\pi}}$, hence by the Stirling formula we have $Q(n,\goodG)\to 1$ as $\nto$.

To finish the proof, it remains to prove the following 
\begin{equation}\label{eq_ratio_badV}
    Q(n,\badG)\to 0\quad \text{as }\nto.
\end{equation}
To prove this, we will use the results of the previous section. Indeed, $n-2\badG=\omega(\log(n))$, hence we just obtained asymptotics for $E(n,\badG)$ ! Due to its technicality, we postpone the proof of~\eqref{eq_ratio_badV} to \Cref{sec:badV}.
\end{proof}

\begin{proposition}\label{prop_s_smallv}
    \Cref{assum_behavior_s} holds and if $g\equiv g_n$ is such that $\smallv$, then $s(n,g)\to 0$ as $\nto$.
\end{proposition}

\begin{proof}
    We can directly calculate that 
\[s(n,g)=\frac{(n-2g-1)(n-2g)}{\log(n)^2}\]
and all the desired properties follow.
\end{proof}

\section{The regime $n-2g=\Theta(\log n)$}\label{sec:midv}

In this section we will prove \Cref{thm_logv}, which describes the asymptotic form of $E(n,g)$ in the regime $n-2g=\Theta(\log n)$. In this regime our method used in previous sections does not apply, since the random walk is not attracted to any boundary.
However, it is possible to obtain the asymptotics using the saddle point method; see~\cite{Gao_log,logn-vertices} for related results.

We start from a representation of $E(n,g)$ as coefficients of a generating function given by Harer and Zagier~\cite{HarerZagier1986Euler} (see also~\cite[Theorem~3.1.5]{LandoZvonkine}):
\begin{align*}
    E(x,y) := 1 + 2xy + 2y \sum_{\substack{n > 0 \\ g\geq 0}} \frac{E(n,g)}{(2n-1)!!} x^{n-2g+1}y^{n} 
    = \left(\frac{1+y}{1-y}\right)^x.
\end{align*}
Hence, $y$ marks the number of edges and $x$ the number of vertices (which is equal to $n-2g+1$ in a unicellular map with $n$ edges).

\begin{remark}
The rescaling by $(2n-1)!!$ originates from the fact that $\sum_{g \geq 0} E(n,g) = (2n-1)!!$, which is the number of perfect matchings of a $2n$-gon (corresponding to a gluing resulting in a unicellular map).
We may thus view $[y^{n+1}]E(x,y)/2$ as the probability generating function of the number of vertices in a uniformly random unicellular map for $n \geq 1$.  
Also note that Recurrence~\eqref{eq:recurrence} follows directly from this closed form. One way to see this is to observe that $E(x,y)$ satisfies the differential equation $(1-y^2)\frac{\partial E}{\partial y} = x E(x,y)$ using the elementary fact that $\left(\frac{1+y}{1-y}\right)^x = \frac{1+y}{1-y}\left(\frac{1+y}{1-y}\right)^{x-1}$.
Using the closed form it is also straightforward to show that the regime $n-2g=\Theta(\log n)$ is the typical regime in the sense that if we draw a unicellular map with $n$ edges uniformly at random it will lie in this regime and therefore have $\Theta(\log(n))$ vertices.
To be more precise: let $\mathbb{E}[u^{X_n}] = [y^{n+1}] E(x,y)/2$. Then Chebyshev's inequality shows that $\frac{X_n}{\log(n)} \stackrel{P}{\rightarrow} 1$ for $n \to \infty$.
\end{remark}

We are interested in $E(n,g) = \frac{(2n-1)!!}{2} [x^{n-2g+1} y^{n+1}] E(x,y)$.
The idea behind our proof is to first determine the asymptotics of $[y^{n+1}] E(x,y)$, using singularity analysis~\cite{flajolet}, then extract the coefficient of $x^m$ from the resulting expression. However, we need to be careful that the error terms appearing in the first step are uniformly bounded, and therefore don't affect the asymptotics in the end. To make this precise, we start with the following lemma.

\begin{lemma}
\label{lem:Exy-firststep-singana}
Fix $r>0$ and let $\cball_r=\{x\in\mathbb C: |x|\leq r\}$. For $x\in \cball_r$,
\begin{align*}
    [y^N]\left(\frac{1+y}{1-y}\right)^x
        &= 2^x\frac{N^{x-1}}{\Gamma(x)} + (-1)^N2^{-x}\frac{N^{-x-1}}{\Gamma(-x)} + R_N(x),
\end{align*}
where
\begin{align*}
    |R_N(x)| \leq C_r N^{|\operatorname{Re}x|-2} 
    % |R_N(x)| \leq C_r\left( N^{\operatorname{Re}x-2} + N^{-\operatorname{Re}x-2} \right)
\end{align*}
for some constant $C_r>0$ independent of $N$ and $x\in \cball_r$.
\end{lemma}

\begin{proof}
The only possible dominant singularities of $E(y) \equiv E(x,y)$ are at $y=1$ and $y=-1$. In fact, $E(y)$ is holomorphic on $\mathbb{C}\setminus\left((-\infty,-1]\cup[1,\infty)\right)$.
In a neighbourhood $y \sim 1$, we get the singular expansion
\begin{align}
    \label{eq:Eyforyat1}
    E(y) = 2^x (1-y)^{-x}\left(1 + \LandauO(1-y)\right),
\end{align}
while in a neighborhood $y \sim -1$, we get
\begin{align}
    \label{eq:Eyforyat-1}
    E(y) = 2^{-x} (1+y)^{x}\left(1 + \LandauO(1+y)\right).
\end{align}
For fixed $x$, the asymptotics of the coefficients is then the sum of the asymptotics of these two singular expansions, however we need to be more careful to make this asymptotic form uniform in $x$. 

%By the transfer theorems~\cite[Section~VI]{flajolet} we get, (for $x\notin\mathbb{Z}$)
%\begin{align}
%    \label{eq:Engyextraction}
%    [y^{N}] E(y) = \frac{2^x N^{x-1}}{\Gamma(x)} + (-1)^{N}\frac{2^{-x} N^{-x-1}}{\Gamma(-x)} + R_N(x),
%\end{align}
%with an error term 
%\[R_N(x)=\LandauO(N^{x-2})+\LandauO(N^{-x-2}).\]
%In other words, for each $x$, there is a constant $C_{x}>0$ satisfying
%\[|R_{N}(x)|\leq C_{x} N^{|\operatorname{Re}x|-2},\]
%for all $N>0$. 
%In fact \eqref{eq:Engyextraction} also holds for $x\in\mathbb{Z}$, where $\Gamma$ evaluated at a negative integer is interpreted as $\infty$. Then it suffices to show that the constant $C_{x}$ coming from the proofs in \cite[Section~VI]{flajolet} can be chosen such that there is a uniform bound $C_{r}$ on all $C_{x}$ for $x\in \cball_r$.
We start with the contribution from the main term in \eqref{eq:Eyforyat1}:
\[[y^{N}]2^x(1-y)^{-x}=2^{x}\frac{\Gamma(x+N)}{\Gamma(x)\Gamma(N+1)}=\frac{2^{x}N^{x-1}}{\Gamma(x)}\left(1+\LandauO\left(\frac{1}{N}\right)\right),\]
where the term $\LandauO\left(\frac{1}{N}\right)$ is uniform for $x\in\cball_r$. 
Similarly, the main term in \eqref{eq:Eyforyat-1} contributes
\[[y^{N}]2^{-x}(1+y)^{x}=(-1)^N\frac{2^{-x}N^{-x-1}}{\Gamma(-x)}\left(1+\LandauO\left(\frac{1}{N}\right)\right).\] 
Note that these expansions hold for $x\in\mathbb{Z}$, where $\Gamma$ evaluated at a negative integer is interpreted as $\infty$.
Moreover, the two contributions to the error term $R_N(x)$ are upper bounded by a constant times $N^{|\operatorname{Re}x|-2}$, as claimed. 
It remains to bound the contributions to $R_N(x)$ arising from the $\LandauO$ terms in \eqref{eq:Eyforyat1} and \eqref{eq:Eyforyat-1}. %of the remainder $R(x,y)$ when these are removed from $E(x,y)$:
%\[R(x,y)=E(x,y)-2^{x}(1-y)^{-x}-2^{-x}(1+y)^{x}.\]
%\todoaepin{Oops, I thought I could simplify things quite a bit with the argument above but then realised it's not so clear that $c_{x}$ is continuous.}
%To do this, we follow the proofs in \cite[Section~VI]{flajolet} and show that the constant $C_{x}$ that appears is uniformly bounded.
%\todoaepin{Summary of ideas in proof:}
%\begin{itemize}
%    \item there is a uniform constant $K_{r}$ such that 
%    \[|2^x (1+y)^{-x-1}E(x,y)-(1+y)^{-1}| <K_{r}.\]
%    This follows from critical behaviour around $y=-1$ and compactness of the domain.\todomwin{why at $y=-1$ for the idea?}
%    \item Using the contour integral from Figure \ref{fig:twodent} on the error term, with the small circles having radius $1/n$ yields the required bound.
%\end{itemize}

%\todoaepin{End of summary}

% \todoaepin{I don't know if the following refences still apply.}
%It remains to prove the uniform bound of the remainder $R_N(x)$.
% Our following reasoning is similar that used in the quasi-powers framework~\cite{Hwang1998CLT} and work related to central limit theorems in combinatorial structures~\cite{FlajoletSoria1990Gaussian}.
The key observation is that the transfer theorems from singularity analysis give a uniform error term provided that $x$ stays in a fixed compact set.
This observation has previously already been used in, e.g., \cite{FlajoletSoria1990Gaussian,Hwang1998CLT}.
% and the same dented domain in the $y$-plane can be used uniformly for all such $x$.
Let us now make this claim precise. 
%To achieve such a bound, we will start by deducing uniform bounds on the function $R(x,y)$.
We will start by showing uniform bounds on the $\LandauO$ terms in the expansions~\eqref{eq:Eyforyat1} and~\eqref{eq:Eyforyat-1} of $E(x,y)$. Writing the $\LandauO(1-y)$ term in~\eqref{eq:Eyforyat1} explicitly as $(1-y)c_{1}(x,y)$, we have
\[c_{1}(x,y)=\frac{(1-y)^{x}E(x,y)-2^{x}}{1-y}=\frac{(1+y)^x-2^x}{1-y}.\]
%\todoaepin{Ah, I didn't do this bit properly.}
%Note that the function
%\[c_{1}(x,y):=R(x,y)(1-y)^{x-1}=\frac{(1+y)^x-2^{x}-2^{-x}(1-y^2)^{x}}{1-y}\]
This has a removable singularity at $y=1$, where we can define $c_{1}(x,1)=-x2^{x-1}$. The resulting function $c_{1}$ is continuous in $x$ and $y$ away from the possible singularity at $y=-1$, so, $c_{1}(x,y)$ is bounded on any space where $x$ and $y$ are bounded, with $y$ not approaching $-1$. %on the compact region
%\[\{(x,y):|x|\leq r,|y|\leq \eta+1,\text{Re}(y)\geq 0\}.\]

%\todoaepin{to change following}

%We start from the expansion~\eqref{eq:Eyforyat1} of $E(x,y)$ at $y \sim 1$. Set $u=1-y$, then 
%\begin{align*}
%    \left( \frac{1+y}{1-y} \right)^x = 2^x u^{-x} \left( 1 - \frac{u}{2}\right)^x.
%\end{align*}
%There exists constants $\delta > 0$ and $C>0$ such that, for all $|x| \leq r$ and $|u| < \delta$ we have
%\begin{align*}
%    \left| \left( 1 - \frac{u}{2}\right)^x - 1 \right| \leq C |u|.
%\end{align*}
%This follows from writing $(1-u/2)^x = \exp(x \log(1-u/2))$ and the fact that $|\log(1-u/2)| \leq 2|u|$ for $|u|<1$.
%Consequently, for $y \sim 1$ we have
Using $c_1(x,y)$ we can express $E(x,y)=2^{x}(1-y)^{-x}+c_{1}(x,y)(1-y)^{1-x}$. 
Hence there is some constant $K>0$ such that for $|x|\leq r$, $|y|\leq \eta+1$ and $\text{Re}(y)\geq 0$, we have
\begin{align*}
    |E(x,y)-2^{x}(1-y)^{-x}|=|c_{1}(x,y)(1-y)^{1-x}| \leq K |(1-y)^{1-x}|.
\end{align*}
Similarly, since $E(x,y)=E(-x,-y)$, for $|x|\leq r$, $|y|\leq \eta+1$ and $\text{Re}(y)\leq 0$, we have
\begin{align*}
    |E(x,y)-2^{-x}(1+y)^x| =|c_{2}(x,y)(1+y)^{1+x}|\leq K |(1+y)^{1+x}|.
\end{align*}
We will use this to bound $[y^N]E(x,y)$ using a contour integral. 
%\todoaepin{Expressions above are no longer just $R(x,y)$ on the left, so we still need to deal with that somehow.}

%uniformly for $|x|\leq r$. 
%Similarly, for $y \sim -1$, we have
%\begin{align*}
%    \left( \frac{1+y}{1-y} \right)^x = 2^{-x} (1+y)^{x} ( 1 + r_2(x,y)) \qquad \text{ with } \qquad |r_2(x,y)| \leq C|1+y|,
%\end{align*}
%uniformly for $|x|\leq r$.

{
\begin{figure}[t]
\centering

% Parameters
\def\R{1.25}      % outer radius = 1+\eta
\def\eps{0.24}    % small circle radius, in the proof eps = 1/N
\def\mphi{25}     % dent half-angle in degrees
\def\r{0.7}       % inner circle radius (representing unit circle)

\begin{tikzpicture}[
    scale=2.5,
    every node/.style={font=\small},
    outer/.style={very thick, blue!70!black, ->},
    side/.style={very thick, orange!85!black, ->},
    small/.style={very thick, red!75!black, ->},
    cut/.style={thin, dashed, black!55},
    boundary/.style={thin, black!70}
]

% Compute intersections of dent sides with outer circle
\pgfmathsetmacro{\dlen}{-\r*cos(\mphi)+sqrt(\R*\R - \r*\r*sin(\mphi)*sin(\mphi))}
\pgfmathsetmacro{\beta}{atan((\dlen*sin(\mphi))/(\r+\dlen*cos(\mphi)))}

% Main points
\coordinate (O)  at (0,0);
\coordinate (Pp) at (\r,0);
\coordinate (Pm) at (-\r,0);

% Outer intersections of right dent
\coordinate (RU) at ({\r+\dlen*cos(\mphi)},{ \dlen*sin(\mphi)});
\coordinate (RL) at ({\r+\dlen*cos(\mphi)},{-\dlen*sin(\mphi)});

% Outer intersections of left dent
\coordinate (LU) at ({-\r-\dlen*cos(\mphi)},{ \dlen*sin(\mphi)});
\coordinate (LL) at ({-\r-\dlen*cos(\mphi)},{-\dlen*sin(\mphi)});

% Small-circle intersections near right dent
\coordinate (ru) at ({\r+\eps*cos(\mphi)},{ \eps*sin(\mphi)});
\coordinate (rl) at ({\r+\eps*cos(\mphi)},{-\eps*sin(\mphi)});

% Small-circle intersections near left dent
\coordinate (lu) at ({-\r-\eps*cos(\mphi)},{ \eps*sin(\mphi)});
\coordinate (ll) at ({-\r-\eps*cos(\mphi)},{-\eps*sin(\mphi)});

% ---------------------------------------------------------------------
% Shaded two-dented domain
% ---------------------------------------------------------------------

% Fill outer disk
\fill[gray!12] (O) circle (\R);

% Remove right dent
\fill[white]
    (Pp) -- (RU)
    arc[start angle=\beta, end angle=-\beta, radius=\R]
    -- cycle;

% Remove left dent
\fill[white]
    (Pm) -- (LU)
    arc[start angle={180-\beta}, end angle={180+\beta}, radius=\R]
    -- cycle;

% Unit circle for reference
\draw[densely dotted, black!45] (O) circle (\r);
\node[black!55] at (\r/1.4+0.15,\r/1.4+0.15) {$|z|=1$};

% Outer boundary
\draw[boundary] (O) circle (\R);
\node[black!65] at (\R/1.4+0.21,\R/1.4+0.21) {$|z|=1+\eta$};

% Dashed dent boundaries
\draw[cut] (Pp) -- (RU);
\draw[cut] (Pp) -- (RL);
\draw[cut] (Pm) -- (LU);
\draw[cut] (Pm) -- (LL);

% Axes
\draw[->, black!45] (-\R-0.1,0) -- (\R+0.1,0) node[right] {$\operatorname{Re} z$};
\draw[->, black!45] (0,-\R-0.1) -- (0,\R+0.1) node[above] {$\operatorname{Im} z$};

% Singularities and origin
\fill[black] (Pp) circle (0.03);
\fill[black] (Pm) circle (0.03);
\fill[black] (O)  circle (0.025);

\node[below] at (Pp) {$1$};
\node[below]  at (Pm) {$-1$};
\node[below left]  at (O)  {$0$};

% ---------------------------------------------------------------------
% Draw radii of small circles that is 1/N
% ---------------------------------------------------------------------
\draw[black] (Pp) -- ({\r+\eps*cos(120)},{\eps*sin(120)});
\node at ({\r-\eps/1.4*3/5},{\eps/1.4/2})  {{\tiny $\frac{1}{N}$}};

\draw[black] (Pm) -- ({-\r+\eps*cos(70)},{\eps*sin(70)});
\node at ({-\r+\eps/1.4*3/5},{\eps/1.4/2})  {{\tiny $\frac{1}{N}$}};

% ---------------------------------------------------------------------
% Contour pieces
% ---------------------------------------------------------------------

% Outer upper arc
\draw[outer]
    (RU) arc[start angle=\beta, end angle={180-\beta}, radius=\R];

% Left upper side of dent
\draw[side] (LU) -- (lu);

% Small circle around -1
\draw[small]
    (lu) arc[start angle={180-\mphi}, end angle={-180+\mphi}, radius=\eps];

% Left lower side of dent
\draw[side] (ll) -- (LL);

% Outer lower arc
\draw[outer]
    (LL) arc[start angle={180+\beta}, end angle={360-\beta}, radius=\R];

% Right lower side of dent
\draw[side] (RL) -- (rl);

% Small circle around +1
\draw[small]
    (rl) arc[start angle={-\mphi}, end angle={-360+\mphi}, radius=\eps];

% Right upper side of dent
\draw[side] (ru) -- (RU);

% ---------------------------------------------------------------------
% Labels
% ---------------------------------------------------------------------

\node[blue!70!black] at (-\R/1.4,\R/1.4+0.22) {$\Gamma_{\mathrm{out}}^{+}$};
\node[blue!70!black] at (-\R/1.4,-\R/1.4-0.17) {$\Gamma_{\mathrm{out}}^{-}$};

\node[orange!85!black] at ({\r+0.50*\dlen*cos(\mphi)},{0.60*\dlen*sin(\mphi)+0.15})
    {$\gamma_{\mathrm{up}}^{(+)}$};
\node[orange!85!black] at ({\r+0.50*\dlen*cos(\mphi)},{-0.60*\dlen*sin(\mphi)-0.15})
    {$\gamma_{\mathrm{low}}^{(+)}$};

\node[orange!85!black] at ({-\r-0.50*\dlen*cos(\mphi)},{0.60*\dlen*sin(\mphi)+0.15})
    {$\gamma_{\mathrm{up}}^{(-)}$};
\node[orange!85!black] at ({-\r-0.50*\dlen*cos(\mphi)},{-0.60*\dlen*sin(\mphi)-0.15})
    {$\gamma_{\mathrm{low}}^{(-)}$};

\node[red!75!black] at (\r-\eps-0.06,0.18) {$\gamma_0^{(+)}$};
\node[red!75!black] at (-\r+\eps+0.12,0.18) {$\gamma_0^{(-)}$};

% Angle markers
\draw[black!60] ({\r/4+3*\R/4},0)
    arc[start angle=0,end angle=\mphi,radius=3*(\R-\r)/4];
\node[black!60] at ({\r/4+3*\R/4+0.05},{3*(\R-\r)/4*sin(\mphi/2)+0.02}) {$\phi$};

\draw[black!60] ({-\r/4-3*\R/4},0)
    arc[start angle=180,end angle={180-\mphi},radius=3*(\R-\r)/4];
\node[black!60] at ({-\r/4-3*\R/4-0.05},{3*(\R-\r)/4*sin(\mphi/2)+0.02}) {$\phi$};

% % Domain label
% \node[gray!60!black] at (0,0.45) {$\Delta^{(2)}(\mphi^{\circ},\eta)$};

\end{tikzpicture}
\caption{A two-dented domain $\Delta^{(2)}(\phi,\eta)$ with singularities at $\pm 1$. The contour is split into outer arcs $\Gamma_{\mathrm{out}}^{\pm}$, rectilinear parts $\gamma_{\mathrm{up}}^{(\pm)}$, $\gamma_{\mathrm{low}}^{(\pm)}$, and small circular arcs $\gamma_0^{(\pm)}$ around the singularities.}
\label{fig:twodent}
\end{figure}
}

We now follow the proof of the transfer theorem of Flajolet and Odlyzko (see \cite[Theorem~1]{FlajoletOdlyzko1990}) in a domain dented at the two singularities $1$ and $-1$. More precisely, fix $\phi\in(0,\pi/2)$. Using the same $\eta>0$ as earlier, we define the two-dented domain
\[
    \Delta^{(2)}(\phi,\eta)
    =
    \left\{
    |y|\leq 1+\eta:
    |\operatorname{arg}(y-1)|\geq \phi,\ 
    |\operatorname{arg}(-y-1)|\geq \phi
    \right\}.
\]
The first condition opens a dent at $1$ in the direction of the cut $[1,\infty)$, and the second opens a dent at $-1$ in the direction of the cut $(-\infty,-1]$; see Figure~\ref{fig:twodent}. Now define $\gamma$ to be the contour in Figure~\ref{fig:twodent} which travels around the boundary of $\Delta^{(2)}(\phi,\eta)$ except around the points $-1$ and $1$ where there are arcs of radius $\frac{1}{N}$. 
Then $[y^{N}]E(x,y)$ can be written in terms of a contour integral
\[[y^{N}]E(x,y)=\frac{1}{2\pi i}\oint_{\gamma}y^{-N-1}E(x,y) \, dy.\]
Then as in the proof of \cite[Theorem~1]{FlajoletOdlyzko1990}, the integral on each section of $\gamma$ can be bounded separately. For example, on the sections $\gamma_{0}^{+}$ and $\gamma_{\text{up}}^{(+)}$ we 
%write $E(x,y)=2^{x}(1-y)^{-x}+2^x (1-y)^{1-x}c_{1}(x,y)$ and 
use our bound on $|c_{1}|$ to show that the error term given by the integral of $y^{-N-1}(E(x,y)-2^{x}(1-y)^{-x})$ on each of these sections is bounded by 
\begin{align*}
%    \left|\oint_{\gamma_{0}^{(+)}}y^{-N-1}2^x (1-y)^{1-x}c_{1}(x,y)dy\right|&\leq 
    \oint_{\gamma_{0}^{+}}K \left|y^{-N-1} (1-y)^{1-x} \right||dy|&\leq K \left(1-\frac{1}{N}\right)^{-N-1} N^{\text{Re}(x)-2}(2\pi-2\phi),\\
    \oint_{\gamma_{\text{up}}^{(+)}}K \left|y^{-N-1} (1-y)^{1-x} \right||dy|&
    %\leq\left|\oint_{\gamma_{\text{up}}^{(+)}}y^{-N-1}2^x (1-y)^{1-x}c_{1}(x,y)dy\right|%\leq \oint_{\gamma_{\text{up}}^{(+)}}\left|y^{-N-1}2^x (1-y)^{1-x}K \right||dy|
    \leq K N^{\text{Re}(x)-2}\int_{1}^{\infty}t^{1-\text{Re}(x)}\left(1+\frac{t\cos(\phi)}{N}\right)^{-N-1}dt.
\end{align*}
We claim that for $|x|\leq r$ and for sufficiently large $N$, both expressions are bounded by $N^{\text{Re}(x)-2}$ times some constant. For the first expression this is clear as $2^{\text{Re}(x)}\leq 2^r$ in this domain, and $\left(1-\frac{1}{N}\right)^{-N-1}\to e$ as $N\to\infty$. To show the same holds for the second expression, it suffices to show that the integral is uniformly bounded (for sufficiently large $N$). Indeed in this domain $t^{1-\text{Re}(x)}\leq t^{1+r}$, so
\[\int_{1}^{\infty}t^{1-\text{Re}(x)}\left(1+\frac{t\cos(\phi)}{N}\right)^{-N-1}dt\leq\int_{1}^{\infty}t^{1+r}\left(1+\frac{t\cos(\phi)}{N}\right)^{-N-1}dt.\]
It is shown in the proof of \cite[Theorem~1]{FlajoletOdlyzko1990}, that the latter integral (called $J_{N}(r,\phi)$) converges to a constant as $N\to\infty$. 
The same argument can be used for the section $\gamma_{\text{low}}^{(+)}$, while the error terms coming from sections $\gamma_{\text{low}}^{(-)}$, $\gamma_{\text{up}}^{(-)}$ and $\gamma_{0}^{(-)}$ are each bounded by $N^{-\text{Re}(x)-2}$ times some constant. 
Finally the contribution from the two sections $\Gamma_{\mathrm{out}}^{+}$ and $\Gamma_{\mathrm{out}}^{-}$ is negligible. 
Combining the bounds shows that $|R_N(x)|$ is bounded by a constant times $N^{|\text{Re}(x)|-2}$.
\end{proof}

We now focus on the asymptotic behavior in terms of $n$, where $N=n+1$. The preceding result can easily be expressed using the following expansions:
\begin{align*}
    N^{x-1} &= n^{x-1} \left(1 + O\left(\frac{1}{n}\right)\right),\\ 
    N^{-x-1} &= n^{-x-1} \left(1 + O\left(\frac{1}{n}\right)\right).
\end{align*}
While the implied constants in these error terms can be chosen uniformly for $\vert{}x\vert{}\leq r$, we will not require this property in what follows. Next, we apply the saddle-point method~\cite[Section~VIII]{flajolet} to determine the asymptotics of the main term. 

\begin{lemma} 
\label{lem:Exy-secondstep-saddlepoint}
For $m,n\to\infty$ satisfying $\frac{m}{\log(n)}\to c$, we have
    \[[x^{m}]\frac{2^x n^{x}}{\Gamma(x)}\sim \frac{2^c}{\Gamma(c)} \frac{(\log n)^{m}}{m!}.\]
\end{lemma}
\begin{proof}
The saddle point method starts, like singularity analysis, with Cauchy's integral representation using a small enough circle around the origin:
\begin{align*}
    [x^m] \frac{2^x n^{x}}{\Gamma(x)} = \frac{1}{2 \pi i} \oint \frac{(2n)^{x}}{\Gamma(x) x^{m+1}} \, dx.
\end{align*}

The main idea is now to use the geometry of the complex integrand as shown in Figure~\ref{fig:saddlepoint}. 
The majority of its mass is centred around a peak, the saddle point, and can be approximated by a Gaussian integral, while the remaining part is negligible.
\begin{figure}
    \centering
    \includegraphics[width=0.5\linewidth]{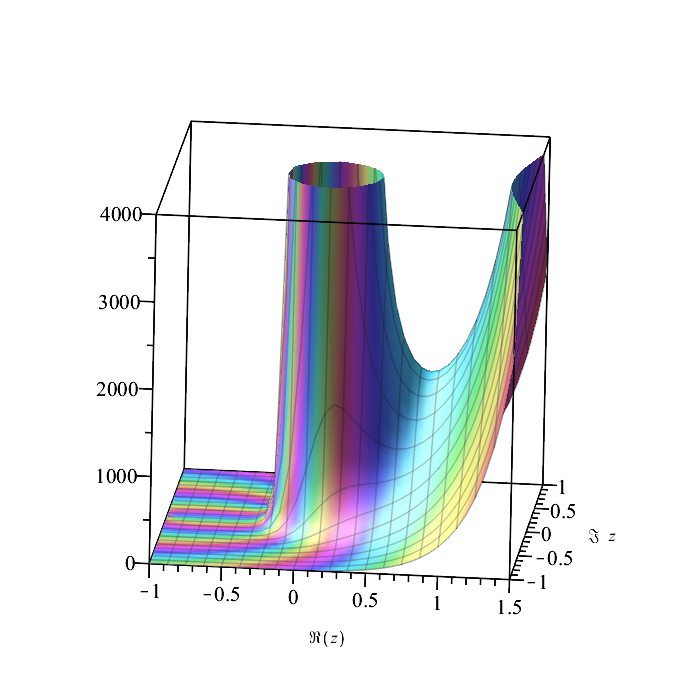}
    \caption{Plot of $|\frac{(2n)^x}{\Gamma(x) x^{m+1}}|$ for $n=1000$ and $m=7 \approx \log(n)$.}
    \label{fig:saddlepoint}
\end{figure}
Our specific expression in~\eqref{eq:labelEngBeforeSP} is a variant of the large-powers framework~\cite[Sections~VIII.8,~IX.5]{flajolet} to quasi-powers (also discussed in \cite[Section~IX.5]{flajolet}, where the results are less precise; see also the original work~\cite{Hwang1998CLT}). In our case we have 
\begin{align*}
    [x^m] \frac{2^x}{\Gamma(x)} e^{x \log(n)} = [x^m] A(x) B(x)^{\log(n)},
\end{align*}
where $A(x)=2^x/\Gamma(x)$ and $B(x)=e^x$.
This falls into the framework of~\cite[Theorem~VIII.8]{flajolet} except that the powers $\log(n)$ are not integers. However, the proof of \cite[Theorem~VIII.8]{flajolet} does not use the assumption that the powers are integers beyond the fact that integer powers produce no new singularities. In our case this is not a problem since $B(x)^{\log(n)} = e^{x \log(n)}$ has no singularities in $\mathbb{C}$.

The other conditions that we need to check to apply \cite[Theorem~VIII.8]{flajolet} are $L_1$, $L_2$, and $L_3$ from the start of \cite[Section~VIII.8.1]{flajolet}. $L_2$ and $L_3$ are easily verified. While $A(x) = \frac{2^x}{\Gamma(x)}$ is analytic at $0$, the second part of condition $L_1$ that $A(x)$ has non-negative coefficients is not satisfied. But, as mentioned in the discussion following the proof in~\cite[page~589]{flajolet}, the weaker assumption that $A(x)$ does not vanish on $(0,\infty)$ suffices. This is of course satisfied for $A(x)$ since it is positive for $x>0$.

Finally, we can apply \cite[Theorem~VIII.8]{flajolet}. First, we compute the spread 
\[
    T = \lim_{x \to \infty} x \frac{B'(x)}{B(x)} = + \infty,
\]
since $B(x)=e^x$.

Next, we set $\lambda = \frac{m}{\log(n)}$. 
Recall that in this case $m \sim c \log(n)$, therefore $\lambda \sim c \in (0,\infty)$. 
The saddle point~$\zeta$ is then the unique root of
\begin{align*}
    \zeta \frac{B'(\zeta)}{B(\zeta)} = \lambda.
\end{align*}
This gives in our case $\zeta=\lambda$.
As a last quantity we need $\xi$ defined by
\begin{align*}
    \xi = \left.\left( \log(B(x)) - \lambda \log(x)\right)''\right|_{x=\zeta}
        = \frac{\lambda}{\zeta^2} = \frac{1}{\lambda}.
\end{align*}
Now we are able to use the generalisation of the large-powers result \cite[Theorem~VIII.8]{flajolet} for log-powers (observe how the terms $n$ are replaced by terms $\log n$). We get 
\begin{align*}
    [x^m] A(x) B(x)^{\log n} 
        \sim A(\zeta) \frac{B(\zeta)^{\log n}}{\zeta^{m+1}\sqrt{2 \pi \xi \log n   }}.
\end{align*}
Therefore, in our case we get
\begin{align}
    [x^m] \frac{2^x}{\Gamma(x)} e^{x \log(n)} 
        &\sim \frac{2^\lambda}{\Gamma(\lambda)} \frac{e^{\lambda \log n}}{\lambda^{m+1/2} \sqrt{2 \pi \log n}} \label{eq:saddlemainterm}\\
        &= \frac{2^\lambda}{\Gamma(\lambda)} \frac{e^{m} (\log n)^{m}}{m^{m+1/2} \sqrt{2 \pi}} \notag \\
        &\sim  \frac{2^c}{\Gamma(c)} \frac{(\log n)^{m}}{m!}, \notag
\end{align}
where we used Stirling's formula and $\lambda \sim c$ in the last step.
\end{proof}

Combining the two lemmas above we can complete the proof of \Cref{thm_logv}

\begin{proof}[Proof of \Cref{thm_logv}]
We need to extract the $x^{n-2g+1}$ coefficient from the expression obtained in \Cref{lem:Exy-firststep-singana}. 
As a shorthand we set $m=n-2g+1$, which is the number of vertices of the corresponding map. 
Now, observe that the two first terms are symmetric in $x$ and $-x$. 
Extracting the coefficient of $x^m$ of the term involving $-x$ gives
\begin{align*}
    [x^m](-1)^{n+1}\frac{2^{-x} n^{-x-1}}{\Gamma(-x)} = (-1)^{n+1+m}[x^m]\frac{2^{x} n^{x-1}}{\Gamma(x)} = [x^m] \frac{2^x n^{x-1}}{\Gamma(x)}.
\end{align*}
Therefore, we get with $N=n+1$
\begin{align*}
    E(n,g) = \frac{(2n-1)!!}{n} [x^m] \frac{2^x n^{x}}{\Gamma(x)} + \frac{(2n-1)!!}{2} [x^m]R_N(x),
\end{align*}
where the factor $2$ got cancelled.
Then, by \Cref{lem:Exy-secondstep-saddlepoint} we get 
\begin{align}
    \label{eq:labelEngBeforeSP}
    E(n,g) = (1+o(1))\frac{(2n-1)!!}{n} \frac{2^c}{\Gamma(c)} \frac{(\log n)^{m}}{m!} + \frac{(2n-1)!!}{2} [x^m]R_N(x).
\end{align}
The asymptotic form in \Cref{thm_logv} comes from the main term after further simplification, using the asymptotic relation
$
\frac{(2n-1)!!}{n} \sim \frac{2^n n!}{\sqrt{\pi n^3}},
$
which is an immediate consequence of Stirling’s formula, together with the fact that $c \sim m/\log n$:
\begin{align*}
    \frac{(2n-1)!!}{n} \frac{2^c}{\Gamma(c)} \frac{(\log n)^{m}}{m!} \sim
    \frac{2^{c}}{c\Gamma(c)} \, \frac{1}{\sqrt{\pi}} \, n^{-3/2} \, 2^n \frac{ n! }{(n-2g)!} \, (\log n)^{n-2g}.
\end{align*}

It remains to prove that the error term is negligible.
For this purpose, we combine the saddle-saddle point argument used in \Cref{lem:Exy-secondstep-saddlepoint} with the uniform bounds from \Cref{lem:Exy-firststep-singana}. 
In our specific case, the saddle point $\lambda = \frac{m}{\log(n)} \to c \in (0,\infty)$. 
Fix a constant $b > c$. 
For large~$n$, the integration contour for $x$ used in the saddle point method lies in the compact  region $\cball_r = \{ x : |x|\leq r \}$.
Additionally we choose $0<a<c$ such that $\lambda \in [a,r]$ for all $n$ large enough.

We start again, with Cauchy's coefficient formula:
\begin{align*}
    [x^m] R_N(x) &= \frac{1}{2 \pi} \int_{-\pi}^{\pi} \frac{R_N(\lambda e^{i \theta})}{\lambda^m e^{i m \theta}} \, d\theta.
\end{align*}
Taking absolute values gives
\begin{align*}
    \left| [x^m] R_N(x) \right| &\leq \frac{1}{2 \pi \lambda^m} \int_{-\pi}^{\pi} |R_N(\lambda e^{i \theta})| \, d\theta.
\end{align*}
Using the uniform bound from \Cref{lem:Exy-firststep-singana} together with $\operatorname{Re}(\lambda e^{i\theta}) = \lambda \cos(\theta)$, we get
\begin{align*}
    \left| [x^m] R_N(x) \right| 
        &\leq \frac{C_r}{\pi \lambda^m} \, N^{-2} \int_{-\pi}^{\pi} N^{\lambda |\cos(\theta)|} \, d\theta
        \leq \frac{2C_r}{\lambda^m} \, N^{\lambda-2}, 
\end{align*}
where we applied a supremum bound to the integral.
This yields the final estimate. To facilitate comparison, we substitute $N=n+1$ and express the asymptotics in terms of $n$:
\begin{align*}
    [x^m] R_N(x) %&= O\left( \frac{n^{\lambda -2}}{\lambda^{m}}\right) 
                 = O\left( \frac{\sqrt{\log n}}{n^2} \frac{(\log n )^m}{m!}\right).
\end{align*}
Substituting this back into~\eqref{eq:labelEngBeforeSP} proves that the error term is $o(1)$ relative to the main asymptotic term.
%
% \begin{align*}
%     \left| [x^m] R_N(x) \right| &\leq \frac{C_r}{\pi \lambda^m} \, N^{-2} \int_{-\pi}^{\pi} N^{\lambda |\cos(\theta)|} \, d\theta.
% \end{align*}
% The last integral can now be upper bounded by a supremum bound by $2\pi n^{\lambda}$, however, the classical Laplace method~\cite[Appendix B.6]{flajolet} (which is the predecessor of the saddle-point method) shows that for $N \to \infty$ it holds that
% \begin{align*}
%     \int_{-\pi}^{\pi} N^{\lambda \cos(\theta)} \, d\theta \sim N^{\lambda} \sqrt{ \frac{2 \pi}{\lambda \log N }}.
% \end{align*}
% This yields the final estimate. To facilitate comparison, we substitute $N=n+1$ and express the asymptotics in terms of $n$:
% \begin{align*}
%     [x^m] R_N(x) &= O\left( \frac{n^{\lambda -2}}{\lambda^{m+1/2}\sqrt{ \log n}}\right) \\
%                  &= O\left( \frac{1}{n^2} \frac{(\log n )^m}{m!}\right).
% \end{align*}
% Substituting this result back into~\eqref{eq:labelEngBeforeSP} proves that the error term is $O(n^{-1})$ relative to the main asymptotic term.
\end{proof}

\section{The random walk method: general results}\label{sec:RW_general}

In this section we generalise the second part of our proof, namely checking with the random walk method, to a large set of recurrences, under a set of minimal conditions. In particular, it can be applied as a black box in future works without mentioning random walks at all.

Consider the following problem, where we are given:
\begin{itemize}
    \item a \emph{step set}, i.e. a finite non empty subset $\mathcal S$ of $\mathbb (\mathbb N\times\mathbb N)\setminus \{(0,0)\}$.
    \item \emph{step weights}, i.e. a family of functions $(P_{i,j})_{i,j\in\mathcal S}$ sending positive integers to positive numbers.
\end{itemize}
and we want to obtain asymptotics for the numbers  $(E(n,g))_{n,g\geq 0}$ that satisfy the recurrence
\begin{equation}\label{eq_rec_general}
    E(n,g)=\sum_{i,j\in\mathcal{S}} P_{i,j}(n)E(n-i,g-j)
\end{equation}
together with given initial conditions.

Suppose that we obtain (by heuristic guessing or however else) an \emph{asymptotic guess}, i.e. some nonzero numbers $(\Omega(n,g))_{n,g\geq 0}$. Then, up to showing similar random walk results as in the previous section, we can show that indeed $\Omega(n,g)\sim E(n,g)$ in some regimes.

\begin{theorem}\label{thm_walk_general}
Given a sequence $g\equiv g_n$, assuming that~\Cref{assum_alpha_beta_general,assum_bd_values_general,assum_behavior_s_general,assum_staring_point_general} defined below are satisfied, we have, as $\nto$,
\[\Omega(n,g)\sim E(n,g).\]
\end{theorem}

The proof of this theorem will be very similar to the proof of~\cref{thm_RW_to_asympto}. For the sake of completeness, we will write it in full, but for the sake of readability, this will be done in the appendix (see \cref{sec:RW_general_proof}).

\subsection{Setup}

As before, we can define
\begin{align}
Q(n,g)&:=\frac{\E(n,g)}{\Omega(n,g)}, \\
    \alpha_{i,j}(n,g) &:= P_{i,j}(n)\cdot \frac{\Omega(n-i,g-j)}{\Omega(n,g)}\quad \text{for all }(i,j)\in\mathcal S. \label{eq_def_alpha_general}
\end{align}
Then~\eqref{eq_rec_general} rewrites into
\begin{equation}\label{eq_Q_general}
    Q(n,g)=\sum_{i,j\in\mathcal{S}} \alpha_{i,j}(n,g)Q(n-i,g-j).
\end{equation}

 We will consider subsets of $\ZZ_{\geq0}$. For each such subset $E$, we introduce the \emph{level lines}
\[E(m)=\{(n,g)\in E\text{ s.t. }n+g=m\}.\]

Now, we give ourselves three disjoint subsets of $\ZZ_{\geq0}$: $\mathcal I$ (the inner points), $\goodB$ and $\badB$ (the good and bad boundaries). We require that they satisfy the following conditions.
\begin{enumerate}
    \item Either for all $m$, for every $n^*,g^*\in\mathcal I (m)$
    \begin{equation}\max\{g\mid (n,g)\in \goodB(m)\}<g^*<\min\{g\mid (n,g)\in \badB(m)\}
    \label{eq_C-}\tag{C$-$}\end{equation}
   
    or for all $m$, for every $n^*,g^*\in\mathcal I (m)$
    \begin{equation}\max\{g\mid (n,g)\in \badB(m)\}<g^*<\min\{g\mid (n,g)\in \goodB(m)\}.\label{eq_C+}\tag{C$+$}\end{equation} We stress that that the chosen inequality must be the same for all $m$.
    \item For all $(n,g)\in \mathcal I$ and $(i,j)\in \mathcal{S}$, $(n-i,g-j)\in \mathcal I \cup\goodB\cup\badB$.
\end{enumerate}

\begin{remark}
    We need a more general notion of boundaries than in \cref{sec:RW}. Indeed, the step set can include large steps, and we want to allow for the case where one of the boundaries is $n=cst$ (which happens if $(0,1)$ belongs to the step set for instance).
\end{remark}

We can define the shifted boundaries as the set of points that are ``one step away'' from a boundary.
\[\goodBB=\{(n,g)\in \mathcal I\text{ s.t. }\exists (i,j)\in \mathcal S; (n-i,g-j)\in \goodB\}\]
and 
\[\badBB=\{(n,g)\in \mathcal I\text{ s.t. }\exists (i,j)\in \mathcal I; (n-i,g-j)\in \badB\}\]

Finally, the Omega ratio is defined as before:
\[s(n,g)=\frac{\Omega(n,g-1)}{\Omega(n,g)}\]
if \eqref{eq_C-} holds, and 
\[s(n,g)=\frac{\Omega(n,g+1)}{\Omega(n,g)}\]
if \eqref{eq_C+} holds.

\subsection{Conditions}
One can define an associated random walk as before, but it will only be needed for the proof, it is in some sense ``hidden'' in the conditions. 

\begin{condition}[Good asymptotic probabilities]\label{assum_alpha_beta_general}
There exists a summable function $\eta$ such that, as $m\to\infty$, uniformly for \mstrict, 
\[\sum_{i,j\in\mathcal{S}} \alpha_{i,j}(n,g)=1+O\left(\eta(m)\right).\]
\end{condition}

The second condition describes the asymptotic behaviour at the boundaries. 
Note that we require a weaker condition for the bad boundary than in \Cref{assum_bd_values}.
\begin{condition}[Boundary values]\label{assum_bd_values_general}
    As $m\to\infty$  we have,
    \[\max\{Q(n,g) \mid (n,g)\in \goodB(m)\}\to 1\text{ and }\min\{Q(n,g)\mid (n,g)\in \goodB(m)\}\to 1\]
    and, for all $(n,g)\in \badB$:
    \[Q(n,g)<C\]
    for some constant $C>0$.
\end{condition}

The third condition describes the behaviour of $s$. 
\begin{condition}[Behaviour of $s$]\label{assum_behavior_s_general}
The numbers $s(n,g)$ satisfy:
\begin{itemize}
    \item $s(n,g)>0$ for all $n,g\in \goodBB$;
    \item there exists a constant $c>0$ such that $s(n,g)>c$ for all $n,g\in \badBB$.
\end{itemize}
\end{condition}

Finally, we need a condition on the sequence $g_n$.
\begin{condition}[Good starting point]\label{assum_staring_point_general}
 The sequence $(g_n)_{n\geq 0}$ is such that
 \begin{itemize}
     \item $(n,g_n)\in \mathcal{I}$ for all $n$ large enough;
     \item $s(n,g_n)\to 0$ as $\nto$.
 \end{itemize}
\end{condition}

\section{Conclusion}\label{sec:discussion}

\subsection{Other works about enumeration and random walks}

Several other works\footnote{This list is most probably not exhaustive.} study links between recurrences, random walks, and enumeration:
\begin{itemize}
    \item In~\cite{compacted-tree} and later works, the first, second, and last author studied the asymptotic enumeration of a family of so-called compacted trees using a bivariate recurrence, also interpreted as space-dependent weighted walks as in this work. However, the walk there may leave the boundary, and the asymptotic initial condition there was unknown. Thus, in contrast to this work, no asymptotic equality but a $\Theta$-result was given there and a stretched exponential term appeared.
    The same method was later applied to solve the asymptotic enumeration of families of minimal DFAs, Young tableaux~\cite{banderier2021walls}, and phylogenetic networks~\cite{fuchs2021phylo}, always showing similar behaviour. 
    \item In~\cite{Agg} and two other works, large genus asymptotics of intersection numbers (a geometric quantity of interest) are obtained by ``comparing the coefficients in [some recursive] relations with the jump probabilities of a certain asymmetric simple random walk''.
    \item In~\cite{FlinChassaing}, the authors study the typical path of the random walk defined by well known linear recurrences, such as Pascal's triangle, and prove a scaling limit. 
\end{itemize}

As far as we understand, each of these three works (and the present one) studies a different setting. It is tempting (and rather ambitious) to ask for a general framework for ``asymptotics via random walks''. We conclude by stating how far we think our method applies, and the next steps we wish to take in this research program.

\subsection{Generality of the method}
In this paper, we presented bivariate asymptotics in one concrete case, but our method should apply to a larger class of recurrences associated to bivariate generating functions satisfying a linear ODE. This is in contrast with the ACSV approach, which currently only applies to a restricted class of algebraic generating functions.

More precisely, our approach could work for any numbers $E(n,k)$
with boundary conditions
\[E(0,0)=1 \qquad \text{ and } \qquad E(n,k)=0\text{ if } n<0 \text{ or }k<0,\]
which satisfy a recurrence of the type
\begin{align}\label{eq_generalized}
    A(n)E(n,k)=\sum_{i=1}^C\sum_{j=0}^D P_{i,j}(n)E(n-i,k-j),
\end{align}
where $C$ and $D$ are nonnegative integers, while $A$ and the $P_{i,j}$'s are polynomials.

The guessing part should go the same way as it did in this paper: 
with the ansatz that  $E(n, \theta n)$ behaves like $n^{c\theta n+d} \exp(n f(\theta))J(\theta)$ for some $c,d\in\mathbb{R}$ and some functions $f(\theta),J(\theta)$, \eqref{eq_generalized} gives a differential equation for $f(\theta)$ as in \Cref{sec:guess}, whose solution gives an estimate of $E(n, k)$. Then a differential equation for $J(\theta)$ could be obtained using the subdominant term as in our case. 

Then, for the checking part, one can use our general \cref{thm_RW_to_asympto}, provided that the conditions are satisfied. In upcoming works, we will classify more systematically which recurrences and which asymptotic regimes satisfy the conditions of \cref{thm_RW_to_asympto}, and we are looking for alternative tools when these conditions are not satisfied.

\subsection{Outlook}%\label{sec:further}

Our next goal is to go beyond linear recurrences. For instance, another classical example of recurrence in map enumeration is the following one in \cite{KP-triangulation} for the number $\tau(n, g)$ of triangulations of genus $g$ and $2n$ faces:
\begin{align}
  \begin{split} \label{eq:KP-triangulation}
    (n + 1) \tau(n, g) &= 4n (3n - 2) (3n - 4) \tau(n - 2, g - 1) + 4 (3n - 1) \tau(n - 1, g)\\
                       &+ 4 \sum_{\substack{i+j=n-2\\i,j\geq 0}} \sum_{\substack{g_1+g_2=g\\g_1,g_2\geq 0}} (3i + 2) (3j + 2) \tau(i, g_1) \tau(j, g_2) + 2 \mathbbl{1}_{n=g=1}.
  \end{split}
\end{align}
We observe that this is a quadratic recurrence with a double sum, to which our random guess-and-check method does not entirely apply. Even if we can simplify the double sum by taking some reasonable approximation, at least one convolutive sum remains, and the random walk modelling does not seem adapted here. However, we may still try to do the asymptotic guessing, and we obtain the following using a combination of guessing and heuristic derivation.

\begin{conjecture} \label{conj:triangulation}
  Given a sequence $g\equiv g_{n}$ such that $g_n / n \to\theta \in (0, 1/2)$ as $n\to\infty$, the following asymptotics hold
  \[
    \tau(n, g) \sim \frac{1}{4(3\pi)^{3/2}} n^{2g-5/2} e^{n f\left(\frac{g}{n}\right)} J\left( \frac{g}{n} \right),
  \]
  with $f$ and $J$ defined as follows: For every $\theta\in(0,\frac{1}{2})$, let $h\equiv h(\theta)\in(0,\frac{1}{4})$ be the unique value satisfying
  \[\theta = \frac{1}{2} - \frac{3h}{(1 + 8h)\sqrt{1 - 4h}} \log \left(\frac{1 + \sqrt{1 - 4h}}{1 - \sqrt{1 - 4h}}\right).\]
  Then we define
  \begin{align*}
    f(\theta) &= 2\theta \log\left(\frac{6h}{(1 + 8h)\sqrt{1 - 4h}}\right) - 2\theta - \log\left(\frac{h}{(1 + 8h)^{3/2}}\right), \\
    J(\theta) &= \frac{(1-4h)(1+8h)^{5/2}}{h^{3/2}\sqrt{(1-2\theta)(1-4h)^2-12\theta h}}.
  \end{align*}
\end{conjecture}

One may compare \Cref{conj:triangulation} with the expression in \Cref{rem:g-infinite} to see the similar structure of the two bivariate asymptotics. With some extra fix inspired by the results from \cite{constant-tg}, we can write another form of \cref{conj:triangulation} that is valid down to $g_n = o(n)$. We are currently working actively on extending our approach to show some weaker version of \Cref{conj:triangulation} and similar asymptotic map-related enumeration problems from integrable systems (see, e.g., \cite{DYZ}).

In another direction, we have been made aware of a problem inspired by population genetics which models the probability of extinction of a species by a random walk in 2D with weighted steps~\cite{flower}. Our method allows us to conjecture the correct bivariate asymptotics for these probabilities, and we are trying to prove it.

Finally, we are working on a variation of our method that applies to recurrences with a \emph{bouncy wall}, where the sum in \eqref{eq_generalized} also allows $j<0$ as in \cite{compacted-tree}, meaning that the walk analogously defined to the one in \Cref{sec:RW} can now leave the $x$-axis.
In this case we expect to be able to determine the asymptotics up to some unknown universal constant.

\section*{Acknowledgements}

Andrew Elvey Price and Wenjie Fang were partially supported by ANR CartesEtPlus (ANR-23-CE48-0018) and ANR IsOMA (ANR-21-CE48-0007).
Baptiste Louf was partially supported byANR CartesEtPlus (ANR-23-CE48-0018) and ANR HighGG
(ANR-24-CE40-2078-01).
Michael Wallner was partially supported by the Austrian Science Fund (FWF): P~34142 and AST1535024, as well as the OeAD WTZ project FR~01/2023.
Generative AI was used only in the late writing stages of this project, to help draft \Cref{fig:twodent}, and for literature research around \Cref{sec:midv}.

\printbibliography

@article {FlajoletOdlyzko1990,
    AUTHOR = {Flajolet, Philippe and Odlyzko, Andrew},
     TITLE = {Singularity analysis of generating functions},
   JOURNAL = {SIAM J. Discrete Math.},
  FJOURNAL = {SIAM Journal on Discrete Mathematics},
    VOLUME = {3},
      YEAR = {1990},
    NUMBER = {2},
     PAGES = {216--240},
      ISSN = {0895-4801},
   MRCLASS = {05A15 (30E20 40E05 41A60)},
  MRNUMBER = {1039294},
MRREVIEWER = {E.\ Rodney\ Canfield},
       DOI = {10.1137/0403019},
       URL = {https://doi.org/10.1137/0403019},
}

@article {FlajoletSoria1990Gaussian,
    AUTHOR = {Flajolet, Philippe and Soria, Mich\`ele},
     TITLE = {Gaussian limiting distributions for the number of components
              in combinatorial structures},
   JOURNAL = {J. Combin. Theory Ser. A},
  FJOURNAL = {Journal of Combinatorial Theory. Series A},
    VOLUME = {53},
      YEAR = {1990},
    NUMBER = {2},
     PAGES = {165--182},
      ISSN = {0097-3165,1096-0899},
   MRCLASS = {05A15 (60C05 60F05)},
  MRNUMBER = {1041444},
MRREVIEWER = {E.\ Rodney\ Canfield},
       DOI = {10.1016/0097-3165(90)90056-3},
       URL = {https://doi.org/10.1016/0097-3165(90)90056-3},
}

@article {Hwang1998CLT,
    AUTHOR = {Hwang, Hsien-Kuei},
     TITLE = {On convergence rates in the central limit theorems for
              combinatorial structures},
   JOURNAL = {European J. Combin.},
  FJOURNAL = {European Journal of Combinatorics},
    VOLUME = {19},
      YEAR = {1998},
    NUMBER = {3},
     PAGES = {329--343},
      ISSN = {0195-6698,1095-9971},
   MRCLASS = {60C05 (05A99 60F05)},
  MRNUMBER = {1621021},
       DOI = {10.1006/eujc.1997.0179},
       URL = {https://doi.org/10.1006/eujc.1997.0179},
}

@Article{DGZZ,
 Author = {Delecroix, Vincent and Goujard, {\'E}lise and Zograf, Peter and Zorich, Anton},
 Title = {Large genus asymptotic geometry of random square-tiled surfaces and of random multicurves},
 FJournal = {Inventiones Mathematicae},
 Journal = {Invent. Math.},
 ISSN = {0020-9910},
 Volume = {230},
 Number = {1},
 Pages = {123--224},
 Year = {2022},
 Language = {English},
 DOI = {10.1007/s00222-022-01123-y},
 zbMATH = {7585646},
 Zbl = {1498.14074}
}

@Article{ChapuyFerayFusy2013Unicellular,
  author     = {Chapuy, Guillaume and F\'{e}ray, Valentin and Fusy, \'{E}ric},
  title      = {A simple model of trees for unicellular maps},
  journal    = {J. Combin. Theory Ser. A},
  year       = {2013},
  volume     = {120},
  number     = {8},
  pages      = {2064--2092},
  issn       = {0097-3165},
  doi        = {10.1016/j.jcta.2013.08.003},
  fjournal   = {Journal of Combinatorial Theory. Series A},
  mrclass    = {05C10 (05A19)},
  mrnumber   = {3102175},
  mrreviewer = {Anna de Mier},
}

@Article{HarerZagier1986Euler,
  author     = {Harer, John and Zagier, Don},
  title      = {The Euler characteristic of the moduli space of curves},
  journal    = {Invent. Math.},
  year       = {1986},
  volume     = {85},
  number     = {3},
  pages      = {457--485},
  issn       = {0020-9910},
  doi        = {10.1007/BF01390325},
  fjournal   = {Inventiones Mathematicae},
  mrclass    = {32G15 (14H15 57R20)},
  mrnumber   = {848681},
  mrreviewer = {William Abikoff},
}

@Article{flower,
 Author = {Lafitte-Godillon, Pauline and Raschel, Kilian and Tran, Viet Chi},
 Title = {Extinction probabilities for a distylous plant population modeled by an inhomogeneous random walk on the positive quadrant},
 FJournal = {SIAM Journal on Applied Mathematics},
 Journal = {SIAM J. Appl. Math.},
 ISSN = {0036-1399},
 Volume = {73},
 Number = {2},
 Pages = {700--722},
 Year = {2013},
 Language = {English},
 DOI = {10.1137/120864258},
 zbMATH = {6190517},
 Zbl = {1412.92203}
}

@article{DYZ,
author = {Dubrovin, B. and Yang, Di and Zagier, Don},
year = {2017},
pages = {601-633},
title = {Classical {H}urwitz numbers and related combinatorics},
volume = {17},
fjournal = {Moscow Mathematical Journal},
journal = {Mosc. Math. J.},
url = {http://www.mathjournals.org/mmj/2017-017-004/2017-017-004-003.html}
}

@Article{logn-vertices,
  author    = {Chmutov, Sergei and Pittel, Boris},
  title     = {The genus of a random chord diagram is asymptotically normal},
  journal   = {J. Combin. Theory Ser. A},
  year      = {2013},
  volume    = {120},
  number    = {1},
  pages     = {102--110},
  month     = jan,
  issn      = {0097-3165},
  doi       = {10.1016/j.jcta.2012.07.004},
  publisher = {Elsevier BV},
}

@Article{KP-triangulation,
  author    = {Goulden, Ian P. and Jackson, David M.},
  title     = {The KP hierarchy, branched covers, and triangulations},
  journal   = {Adv. Math.},
  year      = {2008},
  volume    = {219},
  number    = {3},
  pages     = {932--951},
  month     = oct,
  issn      = {0001-8708},
  doi       = {10.1016/j.aim.2008.06.013},
  publisher = {Elsevier BV},
}

@online{acsv_web,
  author = {Melczer, Stephen and Pemantle, Robin and Wilson, Marc C.},
  title = {The ACSV project},
  url = {https://acsvproject.com/},
}

@Book{acsv,
  title     = {Analytic Combinatorics in Several Variables},
  publisher = {Cambridge University Press},
  year      = {2024},
  author    = {Pemantle, Robin and Wilson, Mark C. and Melczer, Stephen},
  isbn      = {9781108836623},
  doi       = {10.1017/9781108874144},
}

@Article{budzinski-louf,
  author    = {Budzinski, Thomas and Louf, Baptiste},
  title     = {Local limits of uniform triangulations in high genus},
  journal   = {Invent. Math.},
  year      = {2020},
  volume    = {223},
  number    = {1},
  pages     = {1--47},
  month     = jul,
  issn      = {1432-1297},
  doi       = {10.1007/s00222-020-00986-3},
  publisher = {Springer Science and Business Media LLC},
}

@Article{meso-genus,
  author    = {Curien, Nicolas and Kortchemski, Igor and Marzouk, Cyril},
  title     = {The mesoscopic geometry of sparse random maps},
  journal   = {Journal de l’École polytechnique — Mathématiques},
  year      = {2022},
  volume    = {9},
  pages     = {1305--1345},
  month     = sep,
  issn      = {2270-518X},
  doi       = {10.5802/jep.207},
  publisher = {Cellule MathDoc/CEDRAM},
}

@Article{linear-genus,
  author    = {Angel, Omer and Chapuy, Guillaume and Curien, Nicolas and Ray, Gourab},
  title     = {The local limit of unicellular maps in high genus},
  journal   = {Electron. Commun. Probab.},
  year      = {2013},
  pages = {no. 86, 8},
  volume    = {18},
  issn      = {1083-589X},
  doi       = {10.1214/ecp.v18-3037},
  publisher = {Institute of Mathematical Statistics},
}

@article{FlinChassaing,
 author = {Chassaing, Philippe and Flin, Jules and Zevio, Alexis},
 title = {Pascal's formulas and vector fields},
 fjournal = {The Electronic Journal of Combinatorics},
 journal = {Electron. J. Comb.},
 issn = {1077-8926},
 volume = {32},
 number = {1},
 pages = {research paper p1.38, 36},
 year = {2025},
 language = {English},
 doi = {10.37236/12591},
 zbMATH = {8016191}
}

@Article{compacted-tree,
  author    = {Elvey Price, Andrew and Fang, Wenjie and Wallner, Michael},
  title     = {Compacted binary trees admit a stretched exponential},
  journal   = {J. Combin. Ser. A},
  year      = {2021},
  volume    = {177},
  pages     = {105306},
  month     = jan,
  issn      = {0097-3165},
  doi       = {10.1016/j.jcta.2020.105306},
  publisher = {Elsevier BV},
}

@Article{Agg,
 Author = {Aggarwal, Amol},
 Title = {Large genus asymptotics for intersection numbers and principal strata volumes of quadratic differentials},
 FJournal = {Inventiones Mathematicae},
 Journal = {Invent. Math.},
 ISSN = {0020-9910},
 Volume = {226},
 Number = {3},
 Pages = {897--1010},
 Year = {2021},
 Language = {English},
 DOI = {10.1007/s00222-021-01059-9},
 zbMATH = {7433733},
 Zbl = {1480.14020}
}

@Book{flajolet,
  title     = {Analytic Combinatorics},
  publisher = {Cambridge Univ. Press},
  year      = {2009},
  author    = {Flajolet, Philippe and Sedgewick, Robert},
  isbn      = {9780511801655},
  doi       = {10.1017/cbo9780511801655},
}

@Article{guttmannser,
  author    = {Guttmann, Anthony J.},
  title     = {Series extension: predicting approximate series coefficients from a finite number of exact coefficients},
  journal   = {J. Phys. A},
  year      = {2016},
  volume    = {49},
  number    = {41},
  pages     = {415002},
  doi       = {10.1088/1751-8113/49/41/415002},
  publisher = {IOP Publishing},
}

@Article{fuchs2021phylo,
  author   = {Fuchs, Michael and Yu, Guan-Ru and Zhang, Louxin},
  title    = {On the asymptotic growth of the number of tree-child networks},
  journal  = {European J. Combin.},
  year     = {2021},
  volume   = {93},
  pages    = {103278},
  issn     = {0195-6698},
  doi      = {10.1016/j.ejc.2020.103278},
  fjournal = {European Journal of Combinatorics},
}

@article{Gao_log,
 author = {Gao, Zhicheng},
 title = {The genus distribution of cubic graphs and asymptotic number of rooted cubic maps with high genus},
 fjournal = {The Electronic Journal of Combinatorics},
 journal = {Electron. J. Comb.},
 issn = {1077-8926},
 volume = {31},
 number = {2},
 pages = {research paper p2.49, 20},
 year = {2024},
 language = {English},
 doi = {10.37236/11533},
 zbMATH = {7882970},
 Zbl = {1543.05034}
}

@article {banderier2021walls,
    AUTHOR = {Banderier, Cyril and Wallner, Michael},
     TITLE = {Young tableaux with periodic walls: counting with the density
              method},
   JOURNAL = {S\'{e}m. Lothar. Combin.},
  FJOURNAL = {S\'{e}minaire Lotharingien de Combinatoire},
    VOLUME = {85B},
      YEAR = {2021},
     PAGES = {Art. 47, 12},
   MRCLASS = {05A15 (05A19)},
  MRNUMBER = {4311928},
       URL = {https://www.mat.univie.ac.at/~slc/wpapers/FPSAC2021/47.html},
}

@Article{constant-tg,
  author  = {Bender, Edward A. and Gao, Zhicheng and Richmond, L. Bruce},
  title   = {The map asymptotics constant $t_g$},
  journal = {Electron. J. Comb.},
  year    = {2008},
  volume  = {15},
  number  = {1},
  pages   = {R51},
  doi     = {10.37236/775},
}

@Book{acsv-mishna,
  title     = {Analytic Combinatorics: A Multidimensional Approach},
  publisher = {Chapman and Hall/CRC},
  year      = {2019},
  author    = {Mishna, Marni},
  isbn      = {9781351036825},
  doi       = {10.1201/9781351036825},
}

@InProceedings{minimal-automata,
  author    = {Elvey Price, Andrew and Fang, Wenjie and Wallner, Michael},
  title     = {Asymptotics of Minimal Deterministic Finite Automata Recognizing a Finite Binary Language},
  booktitle = {31st International Conference on Probabilistic, Combinatorial and Asymptotic Methods for the Analysis of Algorithms (AofA 2020)},
  year      = {2020},
  volume    = {159},
  series    = {Leibniz International Proceedings in Informatics (LIPIcs)},
  pages     = {11:1--11:13},
  publisher = {Schloss Dagstuhl – Leibniz-Zentrum für Informatik},
  doi       = {10.4230/LIPICS.AOFA.2020.11},
  journal   = {LIPIcs, Volume 159, AofA 2020},
}

@InBook{planar-map-survey,
  chapter   = {Planar Maps},
  pages     = {335--395},
  title     = {Handbook of Enumerative Combinatorics},
  publisher = {Chapman and Hall/CRC},
  year      = {2015},
  author    = {Schaeffer, Gilles},
  doi       = {10.1201/b18255},
}

@Book{LandoZvonkine,
  title      = {Graphs on surfaces and their applications},
  publisher  = {Springer-Verlag},
  year       = {2004},
  author     = {Lando, S. K. and Zvonkin, A. K.},
  volume     = {141},
  series     = {Encyclopaedia of Mathematical Sciences},
  isbn       = {3-540-00203-0},
  note       = {Appendix by D. B. Zagier},
  mrclass    = {14H55 (05-02 05C10 05C30 05C50 14H10 14H30 32G15)},
  mrnumber   = {2036721 (2005b:14068)},
  mrreviewer = {Athanase Papadopoulos},
  pages      = {xvi+455},
}

@InBook{rec-survey,
  chapter   = {Asymptotic Enumeration Methods},
  pages     = {1063–1229},
  title     = {Handbook of combinatorics (Vol. 2)},
  publisher = {MIT Press},
  year      = {1996},
  author    = {Odlyzko, Andrew M.},
  editor    = {Graham, Ronald L. and Grötschel, Martin and Lovász, László},
}

\clearpage

%%------------------------------------------------------------%%
\appendix

\section{Facts about $f$, $\lambda$ and $J$}\label{sec:facts_f_and_co}

In this section we discuss several important facts about $f$, $\lambda$ and $J$ defined in \Cref{thm_main}.
We start with some technical lemmas whose proofs are given in the companion Maple file.

As a first step, note that using the implicit function theorem it is straightforward to see that $\lambda(\theta)$ defined by \eqref{eq:thetasolution} is on $(0,1/2)$ a differentiable function. 
We will now discuss the behaviour of $\lambda$, $f$, and $J$ in detail. 
In particular, we will need precise information on their behaviour for $\theta \to 0$ and $\theta \to 1/2$.

\begin{lemma}\label{lem_prop_small_theta}
    As $\theta\to 0$, setting $\lambda:=\lambda(\theta)$, one has
\begin{equation}\label{eq_small_theta_lambda}
   \lambda(\theta) = \frac{1}{4}-\frac{3 \theta}{4}+\frac{9 \theta^{2}}{20}+\bigO{\theta^{\frac{5}{2}}}.
\end{equation}
Furthermore, it holds that
\begin{align}
    f(\theta)&=\log4-\theta\log\theta+\theta-\theta\log 12+\frac{27}{10}\theta^2+\bigO{\theta^3},\label{eq_f_as_theta_to_0}\\
    f'(\theta)&=-\log\theta+O(1),\label{eq_fprim_as_theta_to_0}\\
    f''(\theta)&=\Theta(\theta^{-1}),\label{eq_fsec_as_theta_to_0}\\
    f'''(\theta)&=\Theta(\theta^{-2}),\label{eq_fter_as_theta_to_0}
\end{align}
and
\begin{align}
    J(\theta)&=\frac{2}{\sqrt \theta}+\frac{27}{5}\sqrt \theta+\bigO{\theta^{3/2}}\label{eq_J_as_theta_to_0},\\
    J'(\theta)&=\Theta(\theta^{-3/2}),\label{eq_Jprim_as_theta_to_0}\\
    J''(\theta)&=\Theta(\theta^{-7/2}).\label{eq_Jsec_as_theta_to_0}
\end{align}
\end{lemma}

\begin{lemma}\label{lem_prop_large_theta}
    As $\theta\to 1/2$, setting $\lambda:=\lambda(\theta)$, one has
\begin{equation}\label{eq_large_theta_lambda}
   \frac{1}{2}-\theta =-\lambda\log(\lambda)-2\lambda^{2}(1+\log(\lambda))+O(\lambda^{3}\log(\lambda)).
\end{equation}
Furthermore, it holds that
\begin{align}
    f(\theta)&=\log\! \left(2\right)-1+\left(2+2 \log\! \left(\lambda \right)^{2}+\left(2 \log\! \left(2\right)-2\right) \log\! \left(\lambda \right)\right) \lambda +\mathrm{O}\! \left(\lambda^{2}\right),\label{eq_f_as_theta_large}\\
    f'(\theta)&=2\log \lambda+2\log(2)+\bigO{\lambda},\label{eq_fprim_as_theta_large}\\
    f''(\theta)&=\frac{2}{\lambda(1+\log(\lambda))}+\bigO{1}=\bigTheta{\halfminth^{-1}},\label{eq_fsec_as_theta_large}\\
    f'''(\theta)&=\Theta\left(\left(\frac 1 2 -\theta \right)^{-2}\right),\label{eq_fter_as_theta_large}
    \end{align}
and
\begin{align}
    J(\theta)&=\sqrt{-\frac{1}{\log\! \left(\lambda \right)+1}}\, \lambda^{-1}-\frac{\sqrt{-\frac{1}{\log\left(\lambda \right)+1}}}{\log\! \left(\lambda \right)+1}+\mathrm{O}\! \left(\lambda \right),\label{eq_J_as_theta_large}\\
   J'(\theta)&=\frac{2\log(\lambda)+3}{2\lambda^{2}(-\log(\lambda)-1)^{5/2}}+\bigO{\frac{1}{\lambda\log(\lambda)}},\label{eq_Jprim_as_theta_large}\\
   J''(\theta)&=\bigTheta{\frac{\sqrt{-\log \lambda}}{(\lambda\log\lambda)^3}},\label{eq_Jsec_as_theta_large}\\
   \frac{J'(\theta)}{J(\theta)}&=-\frac{1}{\lambda\log(\lambda)}-\frac{1}{2\lambda\log(\lambda)^{2}}+\mathrm{O}\! \left(\frac{1}{\lambda\log(\lambda)^4} \right).\label{eq_Jlogprim_as_theta_large}
    \end{align}    
\end{lemma}

In the following lemma, we make the error terms in the Taylor series explicit.
\begin{lemma}\label{lem_taylor_f_J}
For all $0<\theta<1/2$ (including the cases $\theta=o(1)$ and $\frac 1 2 -\theta =o(1)$), and for all $\delta=o\left(\theta\left(\frac 1 2 -\theta\right)\right)$ it holds that
\begin{align}
\begin{aligned}
f(\theta+\delta) 
    &= f(\theta)+\delta f'(\theta)+O\left(\delta^2\theta^{-1}\left(\frac 1 2 -\theta \right)^{-1}\right)\\
    &= f(\theta)+\delta f'(\theta)+\frac{\delta^2}{2}f''(\theta)+O\left(\delta^3\theta^{-2}\left(\frac 1 2 -\theta \right)^{-2}\right)
\end{aligned}
\label{eq_fshift}
\end{align}
and
\begin{align}
\begin{aligned}
J(\theta+\delta)
    &= J(\theta)\left(1+O\left(\delta^{1}\theta^{-1}\left(\frac 1 2 -\theta \right)^{-1}\right)\right)\\
    &= J(\theta)\left(1+\delta\frac{J'(\theta)}{J(\theta)}+O\left(\delta^{2}\theta^{-2}\left(\frac 1 2 -\theta \right)^{-2}\right)\right).
\end{aligned}
\label{eq_Jshift}
\end{align}

\end{lemma}
\begin{proof}
This follows directly from Taylor expansions for $f$ and $J$ as well as
\Cref{lem_prop_large_theta,lem_prop_small_theta}.
% \Cref{eq_fter_as_theta_large,eq_fter_as_theta_to_0,eq_J_as_theta_large,eq_J_as_theta_to_0,eq_Jsec_as_theta_large,eq_Jsec_as_theta_to_0}.
% equations~\eqref{eq_fter_as_theta_large},~\eqref{eq_fter_as_theta_to_0},~\eqref{eq_J_as_theta_large},~\eqref{eq_J_as_theta_to_0}~\eqref{eq_Jsec_as_theta_large} and~\eqref{eq_Jsec_as_theta_to_0}.
%
\begin{comment}

For the first equation, it can be calculated that
\[f'''(\theta)=\Theta(\theta^{-2})\]
as $\theta\to 0$ and 
\[f'''(\theta)=\Theta\left(\left(\frac 1 2 -\theta \right)^{-2}\right)\]
as $\theta\to 1/2$.
Then, one can apply the formula
\[f(\theta+\delta)=f(\theta)+\delta f'(\theta)+\frac{\delta^2}{2}f''(\theta)+\int_0^\delta \frac{x^2}{2}f'''(\theta+x)dx.\]

For the second equation, it can be calculated that
\[J''(\theta)=\Theta(\theta^{-7/2}) \text{ and } J(\theta)=\Theta(\theta^{-1/2}) \]
as $\theta\to 0$ and 
\[\left(\frac 1 2 -\theta \right)=\Theta(\lambda\log(\lambda)) \text{ and } J''(\theta)=\Theta(\lambda^{-4}\log(\lambda)^{-7/2})\text{ and } J(\theta)=\Theta(\lambda^{-1}\log(\lambda)^{-1/2})\]
as $\theta\to 1/2$.
Then, one can apply the formula
\[J(\theta+\delta)=J(\theta)+\delta J'(\theta)+\int_0^\delta \frac{x^2}{2}J''(\theta+x)dx.\]
\end{comment}
\end{proof}

\section{Proof of \Cref{assum_alpha_beta} for \largev}\label{sec:alpha_beta_largev}

Recall that for all integers $n\geq 0$ and $g\leq (n-1)/2$ we define
\begin{align*}
    \Omega(n,g)&:=\frac{1}{2\sqrt{\pi}} \frac{\sqrt{g}(g/e)^g}{g!}n^{2g-2}e^{nf(\frac{g}{n})}{J}\left(\frac{g}{n}\right)K(n-2g)\color{red}{}\color{black} \qquad \text{ with } \\
    K(x)&:=\frac{\sqrt{2\pi} \, x^{x+1}}{e^{x} \, \Gamma(x+3/2)}
\end{align*}
where $f$ and $J$ are the same as in \Cref{thm_main} (recall \Cref{rem_g0} for $g=0$).
Note that $K(x) \rightarrow 1$ as $x\to\infty$.

Now we need to find the asymptotic expansions of $\alpha(n,g) + \beta(n,g)$, defined in~\eqref{eq_alpha_def} and~\eqref{eq_beta_def}, which we recall here:
\begin{align*}
        \alpha(n,g) &= \frac{2(2n-1)}{n+1} \cdot \frac{\Omega(n-1,g)}{\Omega(n,g)},\\
    \beta(n,g) &= \frac{(n-1)(2n-1)(2n-3)}{n+1} \cdot \frac{\Omega(n-2,g-1)}{\Omega(n,g)}.
\end{align*}
The key is to compute the asymptotics of the quotients of $\Omega(n,g)$ appearing in them. 
Let us start by making them more explicit for this specific choice of $\Omega(n,g)$:
\begin{align}
    \frac{\Omega(n-1,g)}{\Omega(n,g)}
        % &= \frac{\frac{1}{2\sqrt{\pi}} \frac{\sqrt{g}(g/e)^g}{g!}(n-1)^{2g-2}e^{(n-1)f(\frac{g}{n-1})}{J}\left(\frac{g}{n-1}\right)K(n-2g-1)}{\frac{1}{2\sqrt{\pi}} \frac{\sqrt{g}(g/e)^g}{g!}n^{2g-2}e^{nf(\frac{g}{n})}{J}\left(\frac{g}{n}\right)K(n-2g)}\\
        % &= \frac{ (n-1)^{2g-2}}{n^{2g-2}} \cdot \frac{{J}\left(\frac{g}{n-1}\right)}{{J}\left(\frac{g}{n}\right)} \cdot \frac{K(n-2g-1)}{ K(n-2g)} e^{(n-1)f(\frac{g}{n-1})-nf(\frac{g}{n})}\\
        &= \left(1-\frac{1}{n}\right)^{2g-2} e^{(n-1)f(\frac{g}{n-1})-nf(\frac{g}{n})} \frac{{J}\left(\frac{g}{n-1}\right)}{{J}\left(\frac{g}{n}\right)} \cdot \frac{K(n-2g-1)}{ K(n-2g)} , \label{eq:Omega_alpha_quot}\\
    \frac{\Omega(n-2,g-1)}{\Omega(n,g)} 
        % &= \frac
        %     {\frac{1}{2\sqrt{\pi}} \frac{\sqrt{g-1}((g-1)/e)^{g-1}}{(g-1)!}(n-2)^{2(g-1)-2}e^{(n-2)f(\frac{g-1}{n-2})}{J}\left(\frac{g-1}{n-2}\right)K(n-2-2(g-1))}
        %     {\frac{1}{2\sqrt{\pi}} \frac{\sqrt{g}(g/e)^g}{g!}n^{2g-2}e^{nf(\frac{g}{n})}{J}\left(\frac{g}{n}\right)K(n-2g)}\\
        % &= \frac
        %     {\frac{\sqrt{g-1}((g-1)/e)^{g-1}}{(g-1)!}(n-2)^{2g-4}e^{(n-2)f(\frac{g-1}{n-2})}{J}\left(\frac{g-1}{n-2}\right)}
        %     {\frac{\sqrt{g}(g/e)^g}{g!}n^{2g-2}e^{nf(\frac{g}{n})}{J}\left(\frac{g}{n}\right)}\\
        % &=  \frac{\sqrt{g-1}}{\sqrt{g}}
        %     \frac{\frac{((g-1)/e)^{g-1}}{(g-1)!}}{\frac{(g/e)^g}{g!}} 
        %     \frac{(n-2)^{2g-4}}{n^{2g-2}} 
        %     e^{(n-2)f(\frac{g-1}{n-2})-nf(\frac{g}{n})}  
        %     \cdot \frac{{J}\left(\frac{g-1}{n-2}\right)}{{J}\left(\frac{g}{n}\right)}\\
        % &=  eg\frac{\sqrt{g-1}}{\sqrt{g}}
        %     \frac{(g-1)^{g-1}}{g^{g}} 
        %     \frac{(n-2)^{2g-4}}{n^{2g-2}} 
        %     e^{(n-2)f(\frac{g-1}{n-2})-nf(\frac{g}{n})}  
        %     \cdot \frac{{J}\left(\frac{g-1}{n-2}\right)}{{J}\left(\frac{g}{n}\right)}\\
        % &=  \frac{e}{(n-2)^2}
        %     \frac{(g-1)^{g-1/2}}{g^{g-1/2}} 
        %     \frac{(n-2)^{2g-2}}{n^{2g-2}} 
        %     e^{(n-2)f(\frac{g-1}{n-2})-nf(\frac{g}{n})}  
        %     \cdot \frac{{J}\left(\frac{g-1}{n-2}\right)}{{J}\left(\frac{g}{n}\right)}\\
        &=  \frac{e}{(n-2)^2}
            \left(1-\frac{1}{g}\right)^{g-1/2}
            \left(1-\frac{2}{n}\right)^{2g-2}
            e^{(n-2)f(\frac{g-1}{n-2})-nf(\frac{g}{n})}  
            \cdot \frac{{J}\left(\frac{g-1}{n-2}\right)}{{J}\left(\frac{g}{n}\right)}. \label{eq:Omega_beta_quot}
\end{align}

We will split the analysis into the following three regimes, which will be treated separately in the following three subsections:
\begin{enumerate}
    \item Very low genus $g<\frac{n^{1/3}}{\log n}$,
    \item Intermediate genus $\frac{n^{1/3}}{\log n}\leq g\leq \frac{n-n^{1/3}}{2}$,
    \item Very high genus $\frac{n-n^{1/3}}{2} < g$.
\end{enumerate}
The key difficulty is to make the error terms precise. 
However, most of the given error terms are not sharp, yet when combining the three regimes the dominant one will be of order $\err$. For this reason, we will give all results up to this order.

\subsection{Very low genus}
We first consider the regime $g<\frac{n^{1/3}}{\log n}$. 
For this section only, we introduce
\[U(n,g):=(g/e)^ge^{nf(\frac{g}{n})} \qquad \text{ and } \qquad V(n,g):=\sqrt{g}J(n/g).\]
By~\eqref{eq_f_as_theta_to_0} and \eqref{eq_J_as_theta_to_0}, respectively, we get, as $\nto$ uniformly for $1\leq g<\frac{n^{1/3}}{\log n}$:
\begin{equation}\label{eq_U_low}
    U(n,g)=\exp\bigpar{n\log4+g\log n-g\log12+\frac{27}{10}\frac{g^2}{n}+\bigO{\frac{g^3}{n^2}}}
\end{equation}
and
\begin{equation}\label{eq_V_low}
    V(n,g)= 2\sqrt{n}+\frac{27}{5}\frac{g}{\sqrt{n}}+\bigO{\frac{g^2}{n^{3/2}}}.
\end{equation}
Now we will estimate $\alpha(n,g)$ and $\beta(n,g)$.
Note that in this regime the error terms are not sharp.
\begin{lemma}
    As $n\to\infty$ and uniformly for $1\leq g<\frac{n^{1/3}}{\log n}$, we have
    \[\alpha(n,g)=1-\frac{3g}{n}+\err.\]
\end{lemma}

\begin{proof}
By~\eqref{eq_U_low}, we have
\[\frac{U(n-1,g)}{U(n,g)}=\frac{1-\frac{g}{n}+\err}{4},\]
and, by~\eqref{eq_V_low},
\[\frac{V(n-1,g)}{V(n,g)}=1-\frac{1}{2n}+\err.\]
It is also easily checked that

\[\frac{K(n-2g-1)}{ K(n-2g)}=1+\err. \]
Therefore, we get the final result
\begin{align*}
    \alpha(n,g)&=\frac{2(2n-1)}{n+1}\frac{(n-1)^{2g-2}}{n^{2g-2}}\frac{U(n-1,g)}{U(n,g)}\frac{V(n-1,g)}{V(n,g)}\frac{K(n-2g-1)}{ K(n-2g)}\\
    &=4\bigpar{1-\frac{3}{2n}}\bigpar{1+\frac{2}{n}-\frac{2g}{n}}\bigpar{\frac{1-\frac{g}{n}}{4}}\bigpar{1-\frac{1}{2n}}+\err\\
    &=1-\frac{3g}{n}+\err. \qedhere
\end{align*}
\end{proof}

\begin{lemma}
    As $n\to\infty$ and uniformly for $1\leq g<\frac{n^{1/3}}{\log n}$, we have
    \[\beta(n,g)=\frac{3g}{n}+\err.\]
\end{lemma}

\begin{proof}
Since $\beta(n,1)$ involves $\Omega(n,0)$ which has a separate definition the case $g=1$ needs to be treated separately, we leave it as an exercise to the reader.

Similarly as in the previous lemma, one can show 
\[\frac{U(n-2,g-1)}{U(n,g)}=\frac{3}{4n}\bigpar{1+\bigO{\frac{g}{n}}}\]
and
\[\frac{V(n-2,g-1)}{V(n,g)}=1+\bigO{\frac{1}{n}}.\]
Now we get
\begin{align*}
    \beta(n,g)&=\frac{(n-1)(2n-1)(2n-3)}{(n+1)(n-2)^2}\frac{(n-2)^{2g-2}}{n^{2g-2}}\frac{gU(n-2,g-1)}{U(n,g)}\frac{V(n-2,g-1)}{V(n,g)}\\
    &\bigpar{4+\bigO{\frac{1}{n}}}\bigpar{1+\bigO{\frac{g}{n}}}\cdot\frac{3g}{4n}\cdot\bigpar{1+\bigO{\frac{g}{n}}}\bigpar{1+\bigO{\frac{1}{n}}}\\
    &=\frac{3g}{n}+\err. \qedhere
\end{align*}
\end{proof}

\subsection{Intermediate genus}\label{sec:intermediate_genus}

We now work in the regime $\frac{n^{1/3}}{\log n}\leq g\leq \frac{n-n^{1/3}}{2}$. All small $o$'s and big $O$'s will be uniform over the whole range. In this section we set 
\[
    v=n-2g \qquad \text{ and } \qquad \theta=\frac{g}{n}.
\]

We start now with an asymptotic expansion for $\alpha(n,g)$.
Note that as in the previous regime the error terms are not sharp.

\begin{lemma}\label{lem_alpha_intermediate}
In the range $\frac{n^{1/3}}{\log n}\leq g\leq \frac{n-n^{1/3}}{2}$, we have, uniformly
    \[\alpha(n,g)= 4\lambda(\theta)\left(1+\frac{1}{n}\left(\frac{\theta J'\left(\theta\right)}{J\left(\theta\right)}+\frac{1}{2}-{\theta}+\frac{\theta^2}{2}f''(\theta)\right)\right)+\err\]
\end{lemma}

\begin{proof}
We need to understand the asymptotics of the terms in~\eqref{eq:Omega_alpha_quot}. 
First, we apply \Cref{lem_taylor_f_J} with $\delta=\frac{g}{n-1}-\frac{g}{n}=\frac{g}{n(n-1)}=\frac{\theta}{n-1}$, where $\theta=\frac{g}{n}$. We get
\begin{align}
(n-1)\f{g}{n-1}
    &= nf(\theta)-f(\theta)+\theta f'(\theta)+\frac{\theta^2}{2(n-1)}f''(\theta)+\bigO{\theta\halfminth^{-2}n^{-2}}\notag\\
    &= nf(\theta)-f(\theta)+\theta f'(\theta)+\frac{\theta^2}{2n}f''(\theta)+\bigO{\frac{1}{v^2}},\label{aaa1}
\end{align}
because a direct computation shows that $1/2-\theta=2v/n = O(1)$, and additionally we have $\theta^2f''(\theta)=O(n/v)$ by~\eqref{eq_fsec_as_theta_to_0} and~\eqref{eq_fsec_as_theta_large}. 
Similarly, we compute
\begin{align}
\frac{J\left(\frac{g}{n-1}\right)}{J(\theta)}&=1+\frac{\theta}{n-1}\frac{J'\left(\theta\right)}{J(\theta)}+\bigO{\halfminth^{-2} n^{-2}}\notag\\
&=\bigpar{1+\frac{\theta}{n}\frac{J'\left(\theta\right)}{J(\theta)}}\bigpar{1+\bigO{\frac{1}{v^2}}}, \label{aaa2}
\end{align}
because  $\frac{\theta J'(\theta)}{J(\theta)}=O(n/v)$ by~\eqref{eq_J_as_theta_to_0},~\eqref{eq_Jprim_as_theta_to_0} for $\theta\to 0$ and by~\eqref{eq_J_as_theta_large},~\eqref{eq_Jprim_as_theta_large} for $\theta\to1/2$.

Also it holds that
\begin{equation}\label{aaa3}
    % \frac{(n-1)^{2g-2}}{n^{2g-2}}
    \left(1-\frac{1}{n}\right)^{2g-2}   
    =\exp(-2\theta)\left(1+\frac{2}{n}-\frac{\theta}{n}\right)\bigpar{1+\bigO{\frac{1}{v^2}}}.
\end{equation}

It is also easily checked that
\[\frac{K(n-2g-1)}{ K(n-2g)}=1+\bigO{\frac{1}{v^2}} \]

So, combining this with~\eqref{aaa1},~\eqref{aaa2} and~\eqref{aaa3}, one gets
\begin{align*}
\alpha(n,g)&=\frac{2(2n-1)}{n+1}\frac{\Omega(n-1,g)}{\Omega(n,g)}\\
&=4\left(1-\frac{3}{2n}\right)\exp(-2\theta)\left(1+\frac{2}{n}-\frac{\theta}{n}\right)\exp\left(-f(\theta)+\theta f'(\theta)+\frac{\theta^2}{2n}f''(\theta)\right)\\
&\quad\times\left(1+\frac{\theta J'\left(\theta\right)}{nJ\left(\theta\right)}\right)\bigpar{1+\bigO{\frac{1}{v^2}}}\\
&=4\lambda(\theta)\left(1+\frac{1}{n}\left(\frac{\theta J'\left(\theta\right)}{J\left(\theta\right)}+\frac{1}{2}-{\theta}+\frac{\theta^2}{2}f''(\theta)\right)\right)\bigpar{1+\bigO{\frac{1}{v^2}}}
\end{align*}
because $\lambda(\theta)=\exp(-2\theta-f(\theta)+\theta f'(\theta))$ due to~\eqref{eq:lambdadef}. 
Finally, we need to show that

\[4\lambda(\theta)\left(1+\frac{1}{n}\left(\frac{\theta J'\left(\theta\right)}{J\left(\theta\right)}+\frac{1}{2}-{\theta}+\frac{\theta^2}{2}f''(\theta)\right)\right)\times \frac{1}{v^2}=\err.\]

But this holds because $\lambda(\theta)=O(n^{-2/3} \log(n)^{-1})$ and $v=O(n^{1/3})$ uniformly in the range we consider, and also that 
\[\frac{1}{n}\left(\frac{\theta J'\left(\theta\right)}{J\left(\theta\right)}+\frac{1}{2}-{\theta}+\frac{\theta^2}{2}f''(\theta)\right)=O(1)\]
once again, uniformly in the range we consider. These properties follow by a quick calculation using \cref{lem_prop_small_theta,lem_prop_large_theta}.
\end{proof}

In a similar fashion we continue with $\beta(n,g)$.

\begin{lemma}\label{lem_beta_intermediate}
In the range $\frac{n^{1/3}}{\log n}\leq g\leq \frac{n-n^{1/3}}{2}$, we have, uniformly
    \[\beta(n,g)=(1-4\lambda(\theta))\left(1+\frac{1}{n}\left(4-4\theta+\frac{(2\theta-1)^2}{2}f''(\theta)+\frac{(2\theta-1)J'\left(\theta\right)}{J\left(\theta\right)}\right)\right)+\err.\]
\end{lemma}

\begin{proof}
For the asymptotics of $\beta(n,g)$ we need to understand the asymptotics of the terms in~\eqref{eq:Omega_beta_quot}. 
First, we apply \Cref{lem_taylor_f_J} with $\delta=\frac{g-1}{n-2}-\frac{g}{n}=\frac{2g-n}{n(n-2)}=\frac{2\theta-1}{n-2}$, where $\theta=\frac{g}{n}$. We get
\begin{align}
(n-2)\f{g}{n-1}&=nf(\theta)-2f(\theta)+(2\theta-1)f'(\theta)+\frac{(2\theta-1)^2}{2(n-2)}f''(\theta)+\bigO{\frac{\halfminth\theta^{-2}}{n^2}}\notag\\
&=nf(\theta)-2f(\theta)+(2\theta-1)f'(\theta)+\frac{(2\theta-1)^2}{2n}f''(\theta)+\bigO{\frac{1}{g^2}},\label{bbb1}
\end{align}
because  $(2\theta-1)^2f''(\theta)=O(n/g)$ by~\eqref{eq_fsec_as_theta_to_0} and~\eqref{eq_fsec_as_theta_large}.

Similarly, we get
\begin{align}
\frac{J\left(\frac{g-1}{n-2}\right)}{J(\theta)}&=1+\frac{2\theta-1}{n-2}\frac{J'\left(\theta\right)}{J(\theta)}+\bigO{\frac{\theta^{-2}}{n^2}}\notag\\
&=\bigpar{1+\frac{2\theta-1}{n}\frac{J'\left(\theta\right)}{J\left(\theta\right)}}\bigpar{1+\bigO{\frac{1}{g^2}}},\label{bbb2}
\end{align}
because  $\frac{(2\theta-1) J'(\theta)}{J(\theta)}=O(n/g)$ 
by~\eqref{eq_J_as_theta_to_0},~\eqref{eq_Jprim_as_theta_to_0} for $\theta\to 0$ and by~\eqref{eq_J_as_theta_large},~\eqref{eq_Jprim_as_theta_large} for $\theta\to1/2$, since
 $\lambda(\theta)=O(n^{-2/3} \log(n)^{-1})$ for $\theta\to1/2$.

Also, a direct computation shows
\begin{equation}\label{bbb3}
   \frac{1}{(n-2)^2}\left(1-\frac{2}{n}\right)^{2g-2}
   =\frac{\exp(-4\theta)}{n^2}\left(1+\frac{8}{n}-\frac{4\theta}{n}\right)\bigpar{1+\bigO{\frac{1}{g^2}}}.
\end{equation}
It is also easily checked that 
\[
    e\left(1-\frac{1}{g}\right)^{g-1/2} = 1+\bigO{\frac{1}{g^2}}.
\]
% \[\frac{\frac{\sqrt{g-1}(g-1/e)^{g-1}}{(g-1)!}}{\frac{\sqrt{g}(g/e)^g}{g!}}=1+\bigO{\frac{1}{g^2}}\]

So, combining this with~\eqref{bbb1},~\eqref{bbb2} and~\eqref{bbb3}, one gets

\begin{align*}
\beta(n,g)
    &=\frac{(n-1)(2n-1)(2n-3)}{n+1}\frac{\Omega(n-2,g-1)}{\Omega(n,g)}\\
    &=4\left(1-\frac{4}{n}\right)\exp(-4\theta)\left(1+\frac{8}{n}-\frac{4\theta}{n}\right)\exp\left[-2f(\theta)+(2\theta-1)f'(\theta)+\frac{(2\theta-1)^2}{2n}f''(\theta)\right]\\
    &\quad\times\left(1+\frac{(2\theta-1)J'\left(\theta\right)}{nJ\left(\theta\right)}\right)\bigpar{1+\bigO{\frac{1}{g^2}}}\\
    &=(1-4\lambda(\theta))\left(1+\frac{1}{n}\left(4-4\theta+\frac{(2\theta-1)^2}{2}f''(\theta)+\frac{(2\theta-1)J'\left(\theta\right)}{J\left(\theta\right)}\right)\right)\bigpar{1+\bigO{\frac{1}{g^2}}},
\end{align*}
because $4\exp(-4\theta-2f(\theta)+(2\theta-1)f'(\theta))=1-4\lambda(\theta)$ due to~\eqref{eq:lambdadef} and~\eqref{eq:nicelambdaf}.
Finally, the error term can be treated in a similar fashion as in the previous proof.
\end{proof}

Finally, we have the following exact expression (see the companion Maple file for a proof):

   \[ 4\lambda(\theta)\left(\frac{\theta J'\left(\theta\right)}{J\left(\theta\right)}+\frac{1}{2}-{\theta}+\frac{\theta^2}{2}f''(\theta)\right)+(1-4\lambda(\theta))\left(
 4 -4\theta+\frac{(2\theta-1)^2}{2}f''(\theta)+\frac{(2\theta-1)J'\left(\theta\right)}{J\left(\theta\right)}\right)=0.\]

Using \Cref{lem_alpha_intermediate,lem_beta_intermediate} we conclude that
\[\alpha(n,g)+\beta(n,g)=1+\err\]
uniformly for all $n,g$ such that $\frac{n^{1/3}}{\log n} \leq g \leq \frac{n-n^{1/3}}{2}$, as $n\to\infty$.

\subsection{Very high genus}

In the section we consider $g\geq \frac{n-n^{1/3}}{2}$. It will be convenient to write the functions in terms of $v=n+1-2g$ and $\gamma=\frac{v-1}{2n}=\frac{1}{2}-\theta$ as then $1<v<n^{1/3}$ in this regime. Then $\gamma$ and $\lambda$ are related by:
\[\gamma=\frac{\lambda}{\sqrt{1-4\lambda}}\log\left(\frac{1+\sqrt{1-4\lambda}}{1-\sqrt{1-4\lambda}}\right).\]
Moreover, we define the function $h(\gamma)$ as the following altered version of $f$:
%recall
%\begin{align*}
%f(\theta)& =  - \theta \log(1-4\lambda) - (1-2\theta) \log(\lambda) + 2(\log(2)-1)\theta,\\
%{J}(\theta)&=\sqrt{\frac{2}{\lambda(1-2\theta-4\lambda+4\theta\lambda)}}.
%\end{align*}
\[h(\gamma):=f\left(\frac{1}{2}-\gamma\right)+1-2\gamma-\log(2)+2\gamma\log(2\gamma)=-\frac{1}{2}\log(1-4\lambda)+2\gamma\log\left(\log\left(\frac{1+\sqrt{1-4\lambda}}{1-\sqrt{1-4\lambda}}\right)\right),\]
and similarly
\[\tilde{J}(\gamma):=\gamma J\left(\frac{1}{2}-\gamma\right)=\frac{\gamma}{\sqrt{\lambda(\gamma-\lambda-2\lambda\gamma})}.\]

Finally we define for $v\geq 2$
%\[K(x)=\frac{\sqrt{2\pi} \, x^{x+1}}{e^{x} \, \Gamma(x+3/2)}\]
\begin{align*}\tilde{\Omega}(n,v)&:=\frac{g!}{\sqrt{2\pi g}(g/e)^{g}}\frac{n!}{\sqrt{2\pi n}(n/e)^{n}}\Omega(n,g)\\
&=\frac{1}{2\sqrt{2}\pi}n^{2g-2}e^{nf(\frac{g}{n})}{J}\left(\frac{g}{n}\right)\cdot\frac{\sqrt{2\pi} \, (v-1)^{v}}{e^{v-1} \, \Gamma(v+1/2)}\\
&=\frac{n!n^{-3/2}2^{n}}{\sqrt{2}\pi\Gamma(v+\frac{1}{2})}e^{nh(\frac{v-1}{2n})}\tilde{J}\left(\frac{v-1}{2n}\right).\end{align*}

From Stirling's approximation, $\tilde{\Omega}(n,v)$ is asymptotically equivalent to $\Omega(n,g)$ for large $n,g$, and more precisely in this regime where $g\sim n$,p
\[\frac{\tilde{\Omega}(n,v)}{\Omega(n,g)}=1+\frac{1}{12n}+\frac{1}{12g}+\bigO{\frac{1}{n^{2}}}.\]
We can therefore approximate $\alpha$ and $\beta$ using $\tilde{\Omega}$ as follows: 
\begin{align}\label{eq:alphainomegatilde}\alpha(n,g)&=\frac{2(2n-1)}{n+1}\frac{\tilde{\Omega}(n-1,v-1)}{\tilde{\Omega}(n,v)}\left(1+O\left(\frac{1}{n^{2}}\right)\right)\\
\label{eq:betainomegatilde}\beta(n,g)&=\frac{(n-1)(2n-1)(2n-3)}{n+1}\frac{\tilde{\Omega}(n-2,v)}{\tilde{\Omega}(n,v)}\left(1+O\left(\frac{1}{n^{2}}\right)\right).\end{align}
We can then rewrite \eqref{eq_J_as_theta_large}, \eqref{eq_Jprim_as_theta_large}, \eqref{eq_Jsec_as_theta_large}, \eqref{eq_fprim_as_theta_large} and\eqref{eq_fsec_as_theta_large}  in terms of $\tilde{J}(\gamma)$ and $h(\gamma)$ as follows

\begin{align}
\tilde{J}(\gamma)
    &= \frac{-\log(\lambda)}{\sqrt{-\log(\lambda)-1}}+\bigO{\lambda\log(\lambda)},
\label{eq_Jtildeser}\\
\tilde{J}'(\gamma)
    &= \frac{\log(\lambda)+2}{\lambda(-\log(\lambda)-1)^{5/2}}+\bigO{\log(\lambda)},
\label{eq_Jprimtildeser}\\
\tilde{J}''(\gamma)
    &= \bigTheta{\frac{\sqrt{-\log \lambda}}{\lambda^2(\log\lambda)^3}},
\label{eq_Jsectildeser}\\
h'(\gamma)&=2\log(-\log(\lambda))+O(\lambda\log(\lambda)).\label{eq_hprime_series}\\
h''(\gamma)&=\frac{2}{\gamma\log(\gamma)}+O(\gamma^{-1}\log(\gamma)^{-3/2}).\label{eq_hsecond_series}
\end{align}
In particular, this implies that as $\gamma>0$ approaches $0$ and $\gamma+\delta=\Theta(\gamma)$, we have stronger versions of \eqref{eq_fshift} and \eqref{eq_Jshift}
\begin{align}h(\gamma+\delta)&=h(\gamma)+\delta h'(\gamma)+O(\delta^{2}\gamma^{-1}\log(\gamma)^{-1}),\label{eq_hshift}\\
\frac{\tilde{J}(\gamma+\delta)}{\tilde{J}(\gamma)}
    &= 1+O\left(\delta\gamma^{-1}\log(\gamma)^{-1}\right)\label{eq_Jtildeshift1}\\
    &= 1+\delta\frac{\tilde{J}'(\gamma)}{\tilde{J}(\gamma)}+O\left(\delta^{2}\gamma^{-2}\log(\gamma)^{-1}\right).\label{eq_Jtildeshift2} \end{align}

Finally by direct calculation (or combining \eqref{eq:lambdadef} with \eqref{eq:nicelambdaf}) we have
\begin{equation}\gamma h'(\gamma)-h(\gamma)=\frac{1}{2}\log(1-4\lambda(\gamma)).\label{eq_hprimeh_exact}\end{equation}
We will use this to bound the values of $\alpha$ and $\beta$.
\begin{lemma}\label{lem_alpha_high}
In the range $\frac{n-n^{1/3}}{2}\leq g \leq \frac{n-2}{2}$, we have, uniformly in $n$
    \[\alpha(n,g)=\frac{2v-1}{n(-\log(\lambda))}+O(1/n/\log(n)^2).\]
\end{lemma}
\begin{proof}
By the expression \eqref{eq:alphainomegatilde}, we have
\begin{equation}\label{eq_alphainhJtilde}\alpha(n,g)=\frac{2v-1}{n}e^{(n-1)h(\frac{v-2}{2n-2})-nh(\frac{v-1}{2n})}\frac{\tilde{J}(\frac{v-2}{2n-2})}{\tilde{J}(\frac{v-1}{2n})}(1+O(n^{-1})).\end{equation}
We will analyse each part of this expression separately. First, writing $\delta=\frac{v-n-1}{2n(n-1)}$, we have
\[h\left(\frac{v-2}{2n-2}\right)=h(\gamma+\delta)=h(\gamma)+\delta h'(\gamma)+O(\delta^{2}\gamma^{-1}\log(\gamma)^{-1}),\]
applying this to the exponent in \eqref{eq_alphainhJtilde} and noting $(n-1)\delta=-\frac{1}{2}+\gamma=\bigO{1}$ yields
\[(n-1)h\left(\frac{v-2}{2n-2}\right)-nh\left(\frac{v-1}{2n}\right)%=(n-1)\delta h'(\gamma)-h(\frac{v-1}{2n})+O(n\delta^{2}\gamma^{-1}\log(\gamma)^{-1})
=\left(\gamma-\frac{1}{2}\right) h'(\gamma)-h\left(\gamma\right)+O(n^{-1}\gamma^{-1}\log(\gamma)^{-1}),\]
which simplifies to the following using \eqref{eq_hprimeh_exact} and \eqref{eq_hprime_series}
\[\frac{1}{2}\log(1-4\lambda)-\log(-\log(\lambda))+O(n^{-1}\gamma^{-1}\log(\gamma)^{-1}).\]
Moreover, in this regime, $\gamma=\frac{v-1}{n}=O(n^{-2/3})$ so $\lambda=o(n^{-2/3})$. Hence we can remove the $\log(1-4\lambda)$ terms from the above expression, leaving us with
\[(n-1)h\left(\frac{v-2}{2n-2}\right)-nh\left(\frac{v-1}{2n}\right)=-\log(-\log(\lambda))+O(n^{-1}\gamma^{-1}\log(\gamma)^{-1})
.\]
Now from \eqref{eq_Jtildeshift1} we have
\[\frac{\tilde{J}(\frac{v-2}{2n-2})}{\tilde{J}(\frac{v-1}{2n})}=1+\bigO{v^{-1}\log(\gamma)^{-1}}\]
Substituting into \eqref{eq_alphainhJtilde} yields
\[\alpha(n,g)=-\frac{2v-1}{\log(\lambda)n}+\bigO{n^{-1}\log(n)^{-2}}.\]
\end{proof}

\begin{lemma}\label{lem_beta_high}
In the range $\frac{n-n^{1/3}}{2}\leq g \leq \frac{n-1}{2}$, we have, uniformly
    \[\beta(n,g)=1+\frac{2v-1}{\log(\lambda)n}+O(\log(n)^{-3} n^{-1}).\]
\end{lemma}
\begin{proof}
From the expression \eqref{eq:betainomegatilde}, we have
\begin{equation}\beta(n,g)=\frac{(2n-1)(2n-3)n^{1/2}}{4(n+1)(n-2)^{3/2}}e^{(n-2)h(\frac{v-1}{2(n-2)})-nh(\frac{v-1}{2n})}\frac{\tilde{J}\left(\frac{v-1}{2(n-2)}\right)}{\tilde{J}\left(\frac{v-1}{2n}\right)}\left(1+\bigO{n^{-2}}\right).\label{eq_betainhJtilde}\end{equation}
We will analyse each part of this expression. First note that the initial fraction of $n$ can be replaced with $1+\bigO{n^{-2}}$. Next, writing $\delta=\frac{2\gamma}{n-2}$, we have from \eqref{eq_Jtildeshift2} that
\[\frac{\tilde{J}\left(\frac{v-1}{2(n-2)}\right)}{\tilde{J}\left(\frac{v-1}{2n}\right)}=\frac{\tilde{J}\left(\gamma+\delta\right)}{\tilde{J}\left(\gamma\right)}=1+\delta\frac{\tilde{J}'(\gamma)}{\tilde{J}(\gamma)}+\bigO{n^{-2}}=1+\frac{1}{\log(\lambda)n}+\bigO{\log(\lambda)^{-3}n^{-1}}.\]
Moreover, from \eqref{eq_hshift}, we have
\[h\left(\frac{v-1}{2(n-2)}\right)=h(\gamma+\delta)=h(\gamma)+\delta h'(\gamma)+\bigO{\delta^{2}\gamma^{-1}\log(\gamma)^{-1}}.\]
So, using $(n-2)\delta=\gamma$, along with \eqref{eq_hprimeh_exact}, we have
\begin{align*}
(n-2)h\left(\frac{v-1}{2(n-2)}\right)-nh\left(\frac{v-1}{2n}\right)
&=2\gamma h'(\gamma)-2h(\gamma)+\bigO{n^{-1}\gamma^{1}\log(\gamma)^{-1}}\\
&=\log(1-4\lambda)+\bigO{n^{-1}\gamma^{1}\log(\gamma)^{-1}}.
\end{align*}
Substituting these expressions into \eqref{eq_betainhJtilde} yields
\begin{align*}
\beta(n,g)
&=\left(1-4\lambda\right)\left(1+\frac{1}{\log(\lambda)n}\right)+O(\log(n)^{-3} n^{-1})\\
&=1+\frac{2v-1}{\log(\lambda)n}+O(\log(n)^{-3} n^{-1}),    
\end{align*}
as required.

\end{proof}

\section{Additional proofs for \largev}

\subsection{Proof of \eqref{eq_bad_largev}}\label{sec:proof_eq_bad_largev}
Here we focus on the case $n=2g+1$, and we want to show \eqref{eq_bad_largev}, which we recall here:

\[ \Omega(2g+1,g)\sim c\left(\frac{4g}{e}\right)^{2g}\log(g)^{3/2}.\]

Set $\theta=\frac{g}{2g+1}$ and $\l=\l(\theta)$. We now find asymptotics for the quantities appearing in the formula for $\Omega$, using \cref{lem_prop_large_theta}.
First we expand $\l$.

\begin{lemma}
    \[\ll=-\log(g)-\log\log(g)-\log(4)+o(1)\]
    and 
    \[-\l\log(\lambda)=\frac{1}{4g}+\smallo{\frac{1}{g\log}}\]
\end{lemma}
\begin{proof}
    We start with the equation
    \[\frac 1 2 -\frac{g}{2g+1}=-\lambda\log(\lambda)+\bigO{\lambda^2\log(\lambda)}\]
    which yields
    \begin{equation}\label{cvb}
        -\l\ll=\frac{1}{4g}+\bigO{\lambda^2\log(\lambda)}+\bigO{\frac{1}{g^2}}
    \end{equation}
    Taking the log in this we get
    \begin{align*}
        \ll&=-\log(g)-\lll-\log(4)+o(1)\\
        &=-\log(g)-\log\Big[\log(g)+\lll+\log(4)+o(1)\Big]-\log(4)+o(1)\\
        &=-\log(g)-\log\log(g)-\log(4)+o(1)
    \end{align*}
    This is the first equation of the lemma, which we can plug back into~\cref{cvb} to get the second equation in the lemma.
\end{proof}

Now we can expand $f$ and $J$.
\begin{lemma}
    \[f(\theta)=\log(2)-1+\frac{1-\log2}{2g}+\frac{\log(g)+\log\log(g)+\log 4}{2g}+\smallo{\frac{1}{g}}\]
    and
    \[J(\theta)\sim 4g\sqrt{\log g}\]
\end{lemma}
\begin{proof}
    We know that
    \[f(\theta)=\log\! \left(2\right)-1+\left(2 \log\! \left(\lambda \right)^{2}+\left(2 \log\! \left(2\right)-2\right) \log\! \left(\lambda \right)\right) \lambda +\mathrm{O}\! \left(\lambda\right)\]
    and we can use the previous lemma to obtain the expansion.

    Similarly,
    \[J(\theta)\sim\sqrt{-\frac{1}{\log\! \left(\lambda \right)}}\, \lambda^{-1}\]
    and using the results of the previous lemma directly yield the desired result
\end{proof}

Now, using the previous lemma, we can prove~\eqref{eq_bad_largev}:
\begin{align*}
    \Omega(2g+1,g)&=\frac{1}{2\sqrt{\pi}} \frac{\sqrt{g}(g/e)^g}{g!}(2g+1)^{2g-2}e^{(2g+1)f(\theta)}{J}\left(\theta\right)K(1)\\
    &\sim cst\cdot (2g)^2g\frac{e}{4g^2}\left(\frac{2}{e}\right)^{2g+1}\frac{e}{2}\cdot g\log(g)\cdot g\sqrt{\log g}\\
    &=c\left(\frac{4g}{e}\right)^{2g}\log(g)^{3/2}
\end{align*}

\subsection{Proof of \Cref{prop_s_largev}}\label{sec:s_largev}

Recall that $\Omega(n,0)=\frac{n^{-3/2}4^n}{\sqrt{\pi}}$. We will use the notation $v=n-2g+1$.

Now, let us estimate $\Omega(n,1)$, using~\eqref{eq_f_as_theta_to_0} and~\eqref{eq_J_as_theta_to_0}

\begin{align*}
    \Omega(n,1)=\frac{1}{2\sqrt{\pi}} (1/e)e^{nf(\frac{1}{n})}{J}\left(\frac{1}{n}\right)K(n-2)\sim \frac{1}{\sqrt{\pi}} \frac{4^nn^{3/2}}{12}
\end{align*}

Hence
\begin{align*}
    s(n,1)=\frac{\Omega(n,0)}{\Omega(n,1)}\sim \frac{12}{n^3}\to 0
\end{align*}
as $n\to \infty$

Now, take $1<g<n/2$. Using~\Cref{lem_taylor_f_J} with $\delta=1/n$, we have

\begin{align*}
    s(n,g)=\frac{\Omega(n,g-1)}{\Omega(n,g)}=\frac{\exp(-f'(\theta))}{n^2}\bigpar{1+\bigO{\frac{\theta^{-1}\halfminth^{-1}}{n}}}\bigO{1}=\bigO{\frac{\exp(-f'(\theta))}{n^2}}.
\end{align*}
We can already say, by~\eqref{eq_fprim_as_theta_to_0}, that if $n\to \infty$ and $g/n\not\to 1/2$, then $s(n,g)\to 0$.

Now we consider the case  $1/2-g/n=o(1)$. Using~\eqref{eq_fprim_as_theta_large} we have
\begin{align*}
    s(n,g)=\bigO{\frac{1}{\lambda^2n^2}}.
\end{align*}
Recall that $v/n\sim -\lambda\log(\lambda)$ by~\eqref{eq_large_theta_lambda} hence for $v\geq 1$, $\log(\lambda)=\bigO{\log n}$. Therefore
\begin{align}\label{ghj}
    s(n,g)=\bigO{\bigpar{\frac{\log(n)}{v}}^2}.
\end{align}

Now we are ready to prove~\cref{prop_s_largev}. First, by definition of $s(n,g)$ and $\Omega(n,g)$, it is easily checked that for all $n\geq 0$ and all $\goodGG\leq g\leq \badGG$, $s(n,g)>0$ (i.e., in that range, the numerator doesn't vanish and the denominator doesn't blow up). By~\eqref{ghj}, we also know that $s(n,\badGG)\to \infty$ as $\nto$. Therefore~\cref{assum_behavior_s} is verified. And the results above show that if $n-2g>>\log n$, then $s(n,g)\to 0$ as $\nto$, thus completing the proof.

\section{Proof of \eqref{eq_ratio_badV}}\label{sec:badV}
Recall that we have $\goodG=n/2$, $\badG=\left\lceil\frac{n-\badV}{2}\right\rceil$, and 
\begin{align}
    \Omega(n,g) &:=\frac{1}{\sqrt{\pi}} n^{-3/2} 2^n \frac{n!}{(n-2g)!}\log(n)^{n-2g}
    % \rightarrow 1\quad\quad\text{as $x\to\infty$},
    \notag
\end{align}

Let $g=\badG$ and $\delta=\log^{1/3}(n)$ so that $n-2g=\log(n)\delta$, we will use the fact that $1<<\delta<<\sqrt{\log n}$.

We have

\begin{align}
\Omega(n,g)&\sim \frac{1}{\sqrt{\pi}}n^{-3/2}2^{n}\frac{n!}{(n-2g)!}\log(n)^{n-2g}\\
&\sim \frac{1}{\sqrt{\pi}}n^{-3/2}2^{n}\frac{n^ne^{n-2g}}{(\log (n) \delta)^{n-2g}e^{n}}\sqrt{\frac{n}{n-2g}}\log(n)^{n-2g}\\
&\sim \frac{1}{\sqrt{\pi}}n^{-3/2}2^{n}\frac{n^ne^{\log(n)\delta}}{\delta^{\log(n)\delta}e^{n}}\sqrt{\frac{n}{\log(n) \delta}}\\
&\sim \frac{n^{n-1+\delta-\delta\log{\delta}}2^{n}}{\sqrt{\pi}\sqrt{\log(n)}e^{n}\sqrt{\delta}}\label{xxx}
\end{align}

On the other hand, let 
\begin{equation}
x:=\frac{n-2g}{2n}=\frac{\log(n)\delta}{2n}
\end{equation}
Also,
\begin{equation}
\log(x)=\log\log(n)+\log(\delta)-\log(2)-\log(n)
\end{equation}

We will write $\l=\l(g/n)$ and 
\[Z:=-\log(\lambda)\]
By~\eqref{eq_large_theta_lambda}, we get
\begin{equation}
x=Z\lambda+o(x^{2-\epsilon})
\end{equation}
Hence, applying log to this
\begin{equation}\label{eq:Z_intermediate}
Z=\log(Z)-\log(x)+o(x^{1-\epsilon})
\end{equation}
By definition of $x$ this already shows
\begin{equation}
Z\sim \log(n).
\end{equation}
So 
\begin{align}
\log(Z)&=\log(\log(Z)-\log(x)+o(x^{1-\epsilon}))\\
&=\log\log(n)+\frac{\log(Z)-\log(xn)}{\log(n)}+O\left(\left(\frac{\log\log(n)}{\log(n)}\right)^2\right)\\
&=\log\log(n)+\frac{\log\log(n)-\log(xn)}{\log(n)}+O\left(\frac{\log\log(n)}{\log(n)^2}\right)
\end{align}
Plugging this back into~\eqref{eq:Z_intermediate} one gets
\begin{align}
Z&=\log(n)+\log\log(n)-\log(xn)\left(1+\frac{1}{\log(n)}\right)+O\left(\frac{\log\log(n)}{\log(n)^2}\right)\\
&=\log(n)+\left(\log(2)-\log(\delta)\right)\left(1+\frac{1}{\log(n)}\right)+O\left(\frac{\log\log(n)}{\log(n)^2}\right)\label{eq:Z_final}
\end{align}
Now,~\eqref{eq_f_as_theta_large} yields
\begin{equation}
f(\frac{g}{n})=(\log(2)-1)(1-2x)+2x\left(Z+\frac 1 Z\right)+o(1/n)
\end{equation}

Putting everything together
\begin{align}
nf(\frac{g}{n})&=(\log(2)-1)(n-\log(n)\delta)+\log(n)\delta[\log(n)+\log(2)-\log(\delta)]+\delta[\log(2)-\log(\delta)]+\delta+o(1)\\
&=(\log(2)-1)n+\log(n)\delta[\log(n)+1-\log(\delta)]+\delta[\log(2)-\log(\delta)]+\delta+o(1)
\end{align}
with gives
\begin{equation}\label{yyy}
\exp(nf(\frac{g}{n}))\sim\left(\frac 2 e\right)^nn^{\log(n)\delta}n^\delta n^{-\delta\log(\delta)}\left(\frac 2 \delta\right)^\delta e^{\delta}
\end{equation}

Also, by~\eqref{eq_J_as_theta_large}
\begin{equation}\label{uuu}
J(g/n)\sim\frac{\sqrt Z}{x}\sim \frac{2n}{\sqrt{\log(n)}\delta}
\end{equation}

And
\begin{equation}\label{eee}
n^{2g}=n^nn^{-\log(n)\delta}
\end{equation}

So starting with \Cref{thm_main} and applying \Cref{yyy,uuu,eee} and in the end \Cref{xxx} we have
\begin{align}
E(n,g)&\sim \frac{1}{2\sqrt 2\pi}n^{2g-2}\exp(nf(g/n))J(g/n)  \\
&\sim\frac{1}{2\sqrt 2\pi}n^{n-2}n^{-\log(n)\delta}\left(\frac 2 e\right)^nn^{\log(n)\delta}n^\delta n^{-\delta\log(\delta)}\left(\frac 2 \delta\right)^\delta e^{\delta}\frac{2n}{\sqrt{\log(n)}\delta}\\
&\sim\frac{1}{\sqrt 2\pi}n^{n-1+\delta-\delta\log(\delta)}\frac{e^{-n+\delta}2^{n+\delta}}{\sqrt{\log(n)}\delta^{\delta+1}}\\
&\sim \Omega(n,g)\bigpar{\frac{2e}{\delta}}^\delta\frac{1}{\sqrt{2\pi\delta}}=o\bigpar{\Omega(n,g)}
\end{align}
and we are done.

\section{Proof of~\Cref{thm_walk_general}\label{sec:RW_general_proof}}
We refer to~\cref{sec:RW_general} for definitions and assumptions.
To simplify the notation, let us introduce
\[A(n,g):=\Sij\alpha_{i,j}(n,g)\]
and the ``error terms''
\begin{align*}
    r^{+}(k)&:=\prod_{j=k+1}^{\infty}\max\left(1,\left\{ A(n,g)  \mid (n,g)\in\mathcal I (k) \right\}\right),\\
    r^{-}(k)&:=\prod_{j=k+1}^{\infty}\min(1,\left\{ A(n,g) \mid (n,g)\in\mathcal I (k) \right\}),
\end{align*}

which due to \Cref{assum_alpha_beta_general} satisfy $r^{-}(k),r^{+}(k)\in(0,\infty)$ and $r^{-}(k),r^{+}(k)\to 1$ as $k\to\infty$.

We consider the random walk $(N_k,G_k)_{k\geq 0}$ started at $(N_0,G_0)=(n,g)$ and stopped as soon as it leaves $\mathcal I$, denoting $\tau=\tau(n,g)$ the stopping time, such that for each step $0\leq k<\tau$, we have, for every $i,j\in\mathcal S$
\[(N_{k+1},G_{k+1})=(N_{k}-i,G_{k}-j)\quad  \text{with probability} \quad \frac{\alpha_{i,j}(N_k,G_k)}{A(N_k,G_k)}.\]

Let $M_k=N_k+G_k$, this is a strictly decreasing quantity of the walk. Let also $Q_k:=Q(N_k,G_k)$, $S_k:=s(N_k,G_k)$, $A_k=A(N_k,G_k)$ and $R^\pm_k=r^\pm(M_k)$.

Once again, we show that $Q_\tau$ and $Q_0$ are not so far from each other.

\begin{proposition}[Conserved quantity]
\label{prop_E(Q)_general}
We have 
\[\mathbb{E}(R^{-}_\tau Q_\tau) \leq R_0^{-}Q_0 \]
and
\[ R_0^{+}Q_0 \leq \mathbb{E}(R^{+}_{\tau}Q_\tau ),\]
\end{proposition}

\begin{proof}
Let us show the first inequality.

For each $k$, we have
\begin{equation}\label{eq_RQ_general}
    \EE(R^-_{k+1}Q_{k+1}\mid N_k,G_k)=\frac{\sum_{i,j\in \mathcal S}\alpha_{i,j}(N_k,G_k)Q(N_k-i,G_k-j)\EE(r^-(M_{k+1})\mid N_k,G_k)}{A_k}.
\end{equation}

By definition of the steps $\mathcal S$, the value $M_k = N_{k}+G_{k}$ is strictly decreasing, and since $r^-$ is increasing, we get
\[\EE(r^-(M_{k+1})\mid N_k,G_k)\leq r^-(M_k-1)\leq R^-_kA_k.\] 
Therefore, recalling~\eqref{eq_Q_general}:
\[\EE(R^-_{k+1}Q_{k+1}\mid N_k,G_k)\leq\sum_{i,j\in \mathcal S}\alpha_{i,j}(N_k,G_k)Q(N_k-i,G_k-j)R_k^-=R_k^-Q_k.\]

Hence $(R_k^-Q_k)_{k \geq 0}$ is a supermartingale and the first inequality follows. Similarly, $(R_k^+Q_k)_{k \geq 0}$ is a submartingale, from which the second inequality follows.

\end{proof}

We also need a similar (but coarser) property on $S_k$:

\begin{lemma}[Conserved quantity, bis]\label{lem_E(S)_general}
Let $S_k:=s(N_k,G_k)$, then
\[\mathbb{E}(S_{\tau-1})=\bigO{S_0}.\]
\end{lemma}

\begin{proof} Within this proof, and for the rest of the section we will write $\pm$ to mean $-$ if \eqref{eq_C-} holds and $+$ if \eqref{eq_C+} holds. So the definition of $s$ can be written as
\[s(n,g)=\frac{\Omega(n,g\pm1)}{\Omega(n,g)}.\]
Combining this with the definition \eqref{eq_def_alpha_general} of the $\alpha_{i,j}$, we have the equality
\begin{align*}
    s(n,g)&=\frac{\Sij \alpha_{i,j}(n,g)s(n-i,g-j)}{A(n,g\pm 1)}\\
    &=\bigpar{1+O(\eta(n+g))}\bigpar{\Sij \alpha_{i,j}(n,g)s(n-i,g-j)}
\end{align*}
where the second equality follows from~\cref{assum_alpha_beta_general}.
Since $\eta$ is summable and $N_k+G_k$ is strictly decreasing, we immediately get that
\[\mathbb{E}(S_k)=O(S_0)\]
for all $k$ which finishes the proof.

\end{proof}
Thanks to these to conserved quantities, we can analyse the behaviour of our random walk.

\begin{proposition}[Behaviour of the random walk]\label{prop_behavior_RW_general}
    The sequence $g\equiv g_n$ of \Cref{thm_RW_to_asympto} is such that for all fixed $L>0$
    \[\P\left((N_\tau,G_\tau)\in\goodB\text{ and } M_\tau>L\right)\to 1\]
    as $\nto$.
\end{proposition}

\begin{proof}
    On the one hand, by \Cref{lem_E(S)_general} and~\cref{assum_staring_point_general}, one has, as $\nto$,
\[\P\left((N_\tau,G_\tau)\in\badB\right)\mathbb{E}\left(S_{\tau-1} \mid (N_\tau,G_\tau)\in\badB\right) \leq \mathbb{E}(S_{\tau-1})=\bigO{S_0}=o(1).\]

But if $(N_\tau,G_\tau)\in\badB$, then necessarily $(N_{\tau-1},G_{\tau-1})\in\badBB$, hence by \Cref{assum_behavior_s_general}  it holds that $S_{\tau-1} >c>0$. Therefore 
\[\mathbb{E}\left(S_{\tau-1} \mid (N_\tau,G_\tau)\in\badB\right)>c>0,\]
and so
\begin{equation}\label{eq_XXX}
    \P\left((N_\tau,G_\tau)\in\badB\right)=o(1).
\end{equation}

On the other hand, fix now a constant $L$. By the same argument one has
\[\P\left(M_\tau\leq L\text{ and }(N_\tau,G_\tau)\in\goodB\right)\mathbb{E}\left(S_{\tau-1} \mid M_\tau\leq L\text{ and }(N_\tau,G_\tau)\in\goodB\right)=o(1).\]
But due to \Cref{assum_behavior_s_general} it holds that \[\mathbb{E}\left(S_{\tau-1} \mid M_\tau\leq L\text{ and }(N_\tau,G_\tau)\in\goodB\right)\geq \min_{1\leq m'\leq L+\max\{i+j\mid (i,j)\in \mathcal S\}\atop (n',g')\in\goodBB(m')}s(n',g')>0.\] 
Therefore one has
\begin{equation}\label{eq_YYY}
   \P\left(M_\tau\leq L\text{ and }(N_\tau,G_\tau)\in\goodB\right)=o(1).
\end{equation}
Finally, \eqref{eq_XXX} and \eqref{eq_YYY} imply the result.
\end{proof}

We are finally ready to prove our general theorem.
\begin{proof}[Proof of \Cref{thm_walk_general}]

    By~\Cref{assum_bd_values_general}, we can bound $Q(n,g)$ for all $(n,g)\in\goodB\cup\badB$, and therefore $Q_\tau\leq C$ deterministically. Hence, by \Cref{prop_behavior_RW_general}, it holds that
    \[\mathbb{E}(R^{\pm}_{\tau}Q_\tau)=\mathbb{E}\left(R^{\pm}_{\tau}Q_\tau \mid  M_\tau\geq L\text{ and }(N_\tau,G_\tau)\in\goodB\right)+o(1).\]
   
    Therefore, by \Cref{prop_E(Q)_general}, we have
    \[Q(n,g)=Q_0\leq\frac{1}{r^+(n+g)}\mathbb{E}(R^{+}_\tau Q_\tau)\leq \frac{1}{r^+(n+g)}\max_{M>L\atop (N,G)\in \goodB(M)}(r^{+}(M)Q(N,G))+o(1).\]
    Note that this holds for any $L$. Using~\Cref{assum_bd_values_general} and remembering that $r^+(k)\to 1$ as $k\to\infty$ we get
    \[Q(n,g)\leq 1+o(1).\]
The lower bound follows in the same way.
\end{proof}

\end{document}